\documentclass[12pt,reqno]{amsart}

\usepackage[T1]{fontenc}
\usepackage[utf8]{inputenc}
\usepackage{amsmath,amsfonts,amsthm,amssymb,amsxtra,bbm, dsfont}
\usepackage{fourier,setspace,graphicx,color,pdflscape}
\usepackage[colorlinks=true]{hyperref}
\usepackage{url}
\hypersetup{urlcolor=blue, linkcolor=blue, citecolor=red, anchorcolor=blue}

\newcounter{Hequation}

\makeatletter\g@addto@macro\equation{\stepcounter{Hequation}}\makeatother

\newcounter{labelrescounter}
\newcommand{\labelres}[1]{\addtocounter{labelrescounter}{1}\hypertarget{res-#1}{\kern 0.1pt}\label{#1}}

\makeatletter
\@addtoreset{equation}{myequation}
\makeatother

\usepackage{etoolbox}
\newcounter{taggedeq}
\pretocmd{\equation}{\stepcounter{taggedeq}}{}{}

\newtheorem{theorem}{Theorem}
\newtheorem{proposition}[theorem]{Proposition}
\newtheorem{lemma}[theorem]{Lemma}
\newtheorem{corollary}[theorem]{Corollary}\newtheorem{remark}[theorem]{Remark}

\newcommand{\innerproduct}[2]{\left\langle #1, #2 \right\rangle}
\newcommand{\R}{{\mathbb R}}

\newcommand{\C}{{\mathbb C}}
\newcommand{\be}[1]{\begin{equation}\label{#1}}
\newcommand{\ee}{\end{equation}}
\renewcommand{\(}{\left(}
\renewcommand{\)}{\right)}

\newcommand{\finprf}{\null\hfill$\square$\vskip 0.3cm}
\renewcommand{\Re}{\mathrm{Re}}

\makeatletter
\@namedef{subjclassname@2020}{%
\textup{2020} Mathematics Subject Classification}
\makeatother
\newcommand{\msc}[1]{\href{https://mathscinet.ams.org/mathscinet/search/mscdoc.html?code=#1}{#1}}

\definecolor{darkgreen}{rgb}{0,0.4,0}

\numberwithin{equation}{section}

\title[Distinguished self-adjoint extension and eigenvalues of operators with gaps]{Distinguished self-adjoint extension and eigenvalues of operators with gaps. Application to Dirac-Coulomb operators}

\author[J.~Dolbeault]{Jean Dolbeault}
\address{\hspace*{-12pt}J.~Dolbeault: CEREMADE (CNRS UMR n$^\circ$ 7534), PSL university, Universit\'e Paris-Dauphine\newline Place de Lattre de Tassigny, 75775 Paris 16, France}
\email{dolbeaul@ceremade.dauphine.fr}

\author[M.J.~Esteban]{Maria J.~Esteban}
\address{\hspace*{-12pt}M.J.~Esteban: CEREMADE (CNRS UMR n$^\circ$ 7534), PSL university, Universit\'e Paris-Dauphine\newline Place de Lattre de Tassigny, 75775 Paris 16, France}
\email{esteban@ceremade.dauphine.fr}

\author[E.~S\'er\'e]{Eric S\'er\'e}
\address{\hspace*{-12pt}E.~S\'er\'e: CEREMADE (CNRS UMR n$^\circ$ 7534), PSL university, Universit\'e Paris-Dauphine\newline Place de Lattre de Tassigny, 75775 Paris 16, France}
\email{sere@ceremade.dauphine.fr}

\begin{document}

\begin{abstract} We consider a linear symmetric operator in a Hilbert space that is neither bounded from above nor from below, admits a block decomposition corresponding to an orthogonal splitting of the Hilbert space and has a variational gap property associated with the block decomposition. A typical example is the Dirac-Coulomb operator defined on $C^\infty_c(\R^3\setminus\{0\}, \C^4)$. In this paper we define a distinguished self-adjoint extension with a spectral gap and characterize its eigenvalues in that gap by a min-max principle. This has been done in the past under technical conditions. Here we use a different, geometric strategy, to achieve that goal by making only minimal assumptions. Our result applied to the Dirac-Coulomb-like Hamitonians covers sign-changing potentials as well as molecules with an arbitrary number of nuclei having atomic numbers less than or equal to 137.\end{abstract}



\maketitle
\thispagestyle{empty}

\section{Introduction and main result}\label{Section:Intro}

In three space dimensions, the {\it free Dirac operator} is of the form $D=-i\,\alpha\cdot\nabla+\beta$ with
\[
 \beta=\left( \begin{matrix} I_2 & 0 \\ 0 & -I_2 \\ \end{matrix} \right),\quad \; \alpha_k=\left( \begin{matrix}
0 &\sigma_k \\ \sigma_k &0 \\ \end{matrix}\right)  \qquad (k=1, 2, 3)\,,
\]
$\sigma_1,\,\sigma_2,\,\sigma_3$ being the Pauli matrices (see~\cite{thaller}). The {\it Dirac-Coulomb operator} is $\,D_V=D+V$ where $V$ is the Coulomb potential $-\frac{\nu}{|x|}$ ($\nu>0$) or, more generally, the convolution of $-\frac{1}{|x|}$ with an extended charge density. Usually, one first defines $D_{-\nu/|x|}$ on the so-called minimal domain $C^\infty_c(\R^3\setminus\{0\},\C^4)$. The resulting minimal operator is symmetric but not closed in the Hilbert space $L^2(\R^3,\C^4)$. It is essentially self-adjoint when $\nu$ lies in the interval $(0,\sqrt{3}/{2}]$. In other words, its closure is self-adjoint and there is no other self-adjoint extension. For larger constants $\nu$ one must define a distinguished, physically relevant, self-adjoint extension and this can be done when $\nu\leq 1$. The essential spectrum of this extension is $\R\setminus (-1,1)$, which is neither bounded from above nor from below. In atomic physics, its eigenvalues in the gap $(-1,1)$ are interpreted as discrete electronic energy levels.
 
Important contributions to the construction of distinguished self-adjoint realisations of Dirac-Coulomb operators were made in the 1970's, see, {\it e.g.},~\cite{Sch72, Wust-73, Wust-75, Wust-77, Nen76, Nen77, KlaWus-78, Klaus-80b}. In these papers, general classes of potentials $V$ are considered, but in the case $V=-\nu/|x|$ one always assumes that $\nu$ is smaller than $1$.
 
Reliable computations of the discrete electronic energy levels in the spectral gap $(-1,1)$ are a central issue in Relativistic Quantum Chemistry. For this purpose, Talman~\cite{Talman-86} and Datta-Devaiah~\cite{datta1988minimax} proposed a min-max principle involving Rayleigh quotients and the decomposition of four-spinors into their so-called large and small two-components. A related min-max principle based on another decomposition using the free-energy projectors $\mathbbm 1_{\,\R_{\pm}}(D)$ was proposed in \cite{esteban1997existence} and justified rigorously in \cite{DES00b} for $\nu\in\left(0,\frac{2}{\pi/2+2/\pi}\right)$.
An abstract version of these min-max principles deals with a self-adjoint operator $A$ defined in a Hilbert space $\mathcal H\,$ and satisfying a {\it variational gap} condition, to be specified later, related to a block decomposition under an orthogonal splitting
\be{Orth}
\mathcal H=\mathcal H_+\oplus\mathcal H_-\,.
\ee
Such an abstract principle was proved for the first time in~\cite{GS99}, but its hypotheses were rather restrictive and the application to the distinguished self-adjoint realization of $D_V$ only gave Talman's principle for bounded electric potentials (see also~\cite{kraus2004variational,tretter2008spectral} for related abstract principles). In~\cite{Griesemer1999}, an improved abstract min-max principle was applied to $D_V$ with the splitting given by the free-energy projectors, for the unbounded potential $-\nu/|x|$ with $\nu\in(0,0.305]$. In~\cite{DES00a}, thanks to a different abstract approach, the range of essential self-adjointness $\nu\in(0,\sqrt{3}/{2}]$ was dealt with, both for Talman's splitting and the free projectors. The articles~\cite{MorMul-15, M16, ELS19, SST20, ELS21} followed and the full range $\nu\in(0,1]$ is now covered.

Using some of the tools of~\cite{DES00a}, Esteban and Loss~\cite{EstLos-07,EstLos-08} proposed a new strategy to build a distinguished, Friedrichs-like, self-adjoint extension of an abstract {\it symmetric} operator with variational gap and applied it to the minimal Dirac-Coulomb operator, with $\nu\in (0, 1]\,$. In~\cite{ELS19, ELS21}, connections were established between this new approach and the earlier constructions for Dirac-Coulomb operators.
 
Important closability  issues had been overlooked in some arguments of~\cite{DES00a} and  some domain invariance questions had not been addressed properly in ~\cite{EstLos-07,EstLos-08} (see the beginning of Subsection~\ref{isom}).  In~\cite{SST20} these issues are clarified and the self-adjoint extension problem considered in~\cite{EstLos-07,EstLos-08} is connected to the min-max principle for the eigenvalues of self-adjoint operators studied in~\cite{DES00a}. The abstract results in~\cite{SST20} have many important applications, but some examples are not covered yet, due to an essential self-adjointness assumption made on one of the blocks. In the corrigendum~\cite{DES2023corrigendum}, we present another way of correcting the arguments of~\cite{DES00a} thanks to a new geometric viewpoint. In the present work, by adopting this viewpoint, we are able to completely relax the essential self-adjointness assumption of~\cite{SST20}. Additionally, our variational gap assumption is more general, as it covers a class of multi-center Dirac-Coulomb Hamiltonians in which the lower min-max levels fall below the threshold of the continuous spectrum (see, e.g., [6] for a study of such operators): we shall use the image that some eigenvalues {\it dive} into the negative continuum. 

Before going into the detail of our assumptions and results, we fix some general notations that will be used in the whole paper. We consider a Hilbert space $\mathcal H$ with scalar product $\innerproduct{\cdot}{\cdot}$ and associated norm $\|\cdot\|$. When the sum $V+W$ of two subspaces $V$, $W$ of $\mathcal H$ is direct in the algebraic sense, we use the notation $V \dot{+} W$. We reserve the notation $V\oplus W$ to topological sums. We adopt the convention of using the same letter to denote a quadratic form $q(\cdot)$ and its polar form $q(\cdot,\cdot)$. We use the notations $\mathcal D(q)$ for the domain of a quadratic form $q\,$, $\mathcal D(L)$ for the domain of a linear operator $L$ and $\mathcal R(L)$ for its range.  The space $\mathcal D(L)$ is endowed with the norm
\[
\|x\|_{\mathcal D(L)}:=\sqrt{\|x\|^2+\|L\,x\|^2}\,,\quad\forall\,x\in\mathcal D(L)\,.
\]
We denote the resolvent set, spectrum, essential spectrum and discrete spectrum of a self-adjoint operator $T$ by $\rho(T)$, $\sigma(T)$, $\sigma_{\rm ess}(T)$ and $\sigma_{\rm disc}(T)$ respectively.

Let us briefly recall the standard Friedrichs extension method. Let $S:\mathcal D(S)\to \mathcal H$ be a densely defined operator. Assume that $S$ is symmetric, which means that $\innerproduct{S\,x}{y}= \innerproduct{x}{S\,y}$ for all $x$, $y\in \mathcal D(S)$. If the quadratic form $s(x)=\innerproduct{x}{S\,x}$ associated to $S$ is bounded from below, {\it i.e.},~if
\[
\ell_1:=\inf_{x\in \mathcal D(S)\setminus\{0\}}\; \frac{s(x)}{\|x\|^2} \;>\;-\,\infty\;,
\]
then $S$ has a natural self-adjoint extension $T$, which is called the {\it Friedrichs extension} of $S$ and can be constructed as follows (see {\it e.g.}~\cite{MR0493420} for more details). First of all, since the quadratic form $s\,$ is bounded from below and associated to a densely defined symmetric operator, it is closable in $\mathcal H$. Denote its closure by $\overline{s}$. Take \hbox{$\ell<\ell_1$}, so that  $\overline{s}(\cdot,\cdot)-\ell\innerproduct{\cdot}{\cdot}$ is a scalar product on $\mathcal D(\overline{s})$ giving it a Hilbert space structure. By the Riesz isomorphism theorem, for each $f\in \mathcal H$, there is a unique $u_f\in\mathcal D(\,\overline{s})$ such that $\overline{s}\,(v,u_f)-\ell\,\innerproduct{v}{u_f}=\innerproduct{v}{f}$ for all $v\in \mathcal D(\,\overline{s})$. Note that $u_f$ is also the unique minimizer of the functional $\mathcal I_f(u):=\frac12\,\big(\,\overline{s}(u)-\ell\,\|u\|^2\big)-\innerproduct{u}{f}$ in $\mathcal D(\,\overline{s})$. The map $f\mapsto u_f$ is linear, bounded and self-adjoint for $\innerproduct{\cdot}{\cdot}$. Its inverse is $T-\ell\,\mathrm{id}_{\mathcal H}$ and one easily checks that $T$ does not depend on $\ell$: this operator is just the restriction of $S^*$ to $\mathcal D(\,\overline{s})\cap \mathcal D(S^*)$. An important property of the Friedrichs extension is that the eigenvalues of $T$ below its essential spectrum, if they exist, can be characterized by the classical Courant-Fisher min-max principle: for every positive integer $k$, the level
\[\ell_k:=\,\inf_{\scriptstyle\begin{array}{c}\mbox{\scriptsize $V$ subspace of $\,\mathcal D(S)$}\\[-4pt] \mbox{\scriptsize dim$\,V=k$}\end{array}}\,\sup_{x\in V\setminus\{0\}}\;\frac{s(x)}{\|x\|^2}\]
is either the bottom of $\sigma_{\!\rm  ess}(T)$ (in the case $\ell_j=\ell_k$ for all $j\geq k$) or the $k$-th eigenvalue of $T$ (counted with multiplicity) below $\sigma_{\!\rm ess}(T)$.

In the special case of the Laplacian in a bounded domain $\Omega$ of~$\R^d$ with smooth boundary, $S=-\Delta:\, C^\infty_c(\Omega)\to L^2(\Omega)$, one has $\mathcal D(\,\overline{s})=H^1_0(\Omega)$ and the construction of the Frie\-drichs extension $T$ corresponds to the weak formulation in $H^1_0(\Omega)$ of the Dirichlet problem: $-\Delta u = f \hbox{ in }\Omega$, $u=0 \hbox{ on }\partial \Omega$. In other words, $u_f$ is the unique function in $H^1_0(\Omega)$ such that for all $v\in H^1_0(\Omega)$, $\int_\Omega \nabla u_f\cdot\nabla v\,dx=\int_\Omega f\,v\,dx$. So $T$ is the self-adjoint realization of the Dirichlet Laplacian. By regularity theory, we learn that $\mathcal D(T)=H^2(\Omega)\cap H^1_0(\Omega)$.

From now on in this paper, we consider a dense subspace $F$ of $\mathcal H$ and a symmetric operator $A:\,F\rightarrow \mathcal H$. We do {\it not} assume that the quadratic form $a(x):=\innerproduct{x}{Ax}$ is bounded from below, so we cannot apply the standard Friedrichs extension theorem to $A$. We introduce an orthogonal splitting $\mathcal H= \mathcal H_+\oplus\mathcal H_-$ of $\mathcal H$ as in~\eqref{Orth}. We denote by
\[
\Lambda_\pm:\mathcal H\rightarrow\mathcal H_\pm
\]
the orthogonal projectors associated to this splitting. We make the following assumptions:
\be{H1}\tag{H1}
\mbox{{\it $F_+:=\Lambda_+F\,$ and $\,F_-:=\Lambda_-F$ are subspaces of $\,F$}}
\ee
and
\be{H2}\tag{H2}
\lambda_0 := \sup_{x_-\in F_-\setminus\{0\}}\frac{a(x_-)}{\|x_-\|^2} < +\infty\;.
\ee
We also make the {\it variational gap assumption} that
\be{H3}\tag{H3}
\mbox{{\it for some $k_0\ge1$, we have $\lambda_{k_0}> \lambda_{k_0-1}=\lambda_0$}}
\ee
where the min-max levels $\lambda_k$ $(k\geq 1)$ are defined by
\be{min-max}
\lambda_k:=\inf_{\begin{array}{c}\scriptstyle V\mbox{\scriptsize subspace of }F_+\\[-2pt]\scriptstyle\mbox{\scriptsize dim}\,V=k\end{array}}\sup_{x\in(V\oplus F_-)\setminus\{0\}}\frac{a(x)}{\|x\|^2}\;.
\ee

\noindent In order to construct a distinguished self-adjoint extension of $A$, for $E>\lambda_0$ we are going to decompose the quadratic form $a-E\Vert\cdot\Vert^2$  as the difference of two quadratic forms $q_E$ and $\overline{b}_E$ with $q_E$ bounded from below and closable, while $\overline{b}_E$ is positive and closed. Before stating our main result, let us define these quadratic forms.\smallskip

\noindent
We first introduce a quadratic form $b$ on $F_-$:
\be{B}
b(x_-)=-a(x_-)=\innerproduct{x_-}{\(-\Lambda_- A\upharpoonright_{_{F_-}}\)x_-}\quad\forall\,x_-\in F_-\,.
\ee
For $E>\lambda_0\,$ it is  convenient to define the associated form
\be{Be}
b_E(x_-)=b(x_-)+E\,\Vert x_-\Vert^2\quad\forall\,x_-\in F_-\,.
\ee
As a consequence of Assumption~\eqref{H2} and of the symmetry of $\,-\Lambda_- A\upharpoonright_{_{F_-}}\,,$ we have that 
\be{closeformb}\tag{$b$}
\mbox{\parbox{10cm}{{\it $b_E$ is positive definite for all $E>\lambda_0$ and $b$ is closable in $\mathcal H_-$.}}}
\ee
 We denote by $\overline{b}$ the closure of $b$ and by $\overline{b}_E=\overline{b}+E\Vert\cdot\Vert^2$ the closure of $b_E$, their domain being $\mathcal D\big(\overline{b}\big)$. We can consider the Friedrichs extension $B$ of $\,-\Lambda_- A\upharpoonright_{_{F_-}}\,$. For every parameter $E>\lambda_0$, the operator $B+E: \mathcal D(B) \rightarrow \mathcal H_-$ is invertible with bounded inverse. This allows us to define the operator $L_E:\,F_+\rightarrow \mathcal D(B)$ such that 
\be{def L}
L_E\,x_+:= (B +E)^{-1}\Lambda_-A\,x_+\,,\quad \forall x_+\in F_+\,.
\ee

\noindent
We then introduce the subspace
\be{def Gamma}
\Gamma_E:= \big\{x_++L_E\,x_+\,:\, x_+\in F_+\big\}\subset F_+\oplus \mathcal D(B)\,.
\ee
Making an abuse of terminology justified by the isomorphism $F_+\oplus \mathcal D(B)\approx F_+\times \mathcal D(B)$, we call $\Gamma_E$ the {\it graph} of $L_E$. On this space, we define a quadratic form $q_E$ by
\be{def Q}
q_E(x_++L_E\,x_+):= \innerproduct{x_+}{(A-E)\,x_+}+\innerproduct{L_E\,x_+}{(B+E)\,L_E\,x_+}\,.
\ee
Denoting by $\overline \Gamma_E$ the closure of $\Gamma_E$ in $\mathcal H$ and by $\Pi_E$ the orthogonal projection on $\overline \Gamma_E$, we may write \[q_E(x)=\innerproduct{x}{S_Ex}\,,\quad \forall x\in \Gamma_E\,,
\]
where
\be{def S}
S_E:= \Pi_E\big( \Lambda_+(A-E)\,\Lambda_+ + \Lambda_-(B+E)\,\Lambda_-\big)\upharpoonright_{_{\Gamma_E}}.
\ee
The operator $S_E$ is symmetric and densely defined in the Hilbert space $\(\overline\Gamma_E\,,\,\innerproduct{\cdot}{\cdot}\!\upharpoonright_{_{\overline\Gamma_E\times \overline\Gamma_E}}\)$. It is one of the two Schur complements associated with the block decomposition of the operator $A-E\,\mathrm{id}_{\mathcal H}$ under the orthogonal splitting $\mathcal H=\mathcal H_+\oplus\mathcal H_-$. 
Further details on $q_E$, $S_E$ are given in Section~\ref{Section:Gen}. In particular, in Subsection~\ref{decomp} the decomposition of $a-E\Vert \cdot\Vert^2$ in terms of $q_E$, $\overline{b}_E$ is given. Note that in \cite{DES00a} (before its corrigendum~\cite{DES2023corrigendum}) as well as in \cite{EstLos-07,EstLos-08,SST20}, the form $\overline{b}$ was already present and a form analogous to $q_E$ was defined, but its domain was $F_+$ instead of $\Gamma_E$.\smallskip

\noindent The main result of this paper is as follows.
\begin{theorem}\label{Theorem:thm1} Let $A$ be a densely defined symmetric operator on the Hilbert space $\mathcal H\,$ with domain $F$. Assume~\eqref{H1}-\eqref{H2}-\eqref{H3} and take $E>\lambda_0$. With the above notations, the quadratic forms $b$ and $q_E$ are bounded from below, $b$ is closable in $\mathcal H_-$, $q_E$ is closable in $\overline \Gamma_E$ and they satisfy
\[
\mathcal D\big(\,\overline{q}_E\big)\cap\mathcal D\big(\,\overline{b}\big)=\big\{0\big\}\,.
\]
The operator $A$ admits a unique self-adjoint extension $\widetilde A$ such that 
\[
\mathcal D\big(\widetilde A\,\big)\subset \mathcal D\big(\,\overline{q}_E\big)\dot{+}\mathcal D\big(\,\overline{b}\big)\,.
\]
The domain of this extension is
\[
\mathcal D\big(\widetilde A\,\big)=\mathcal D(A^*)\cap \Big( \mathcal D\big(\,\overline{q}_E\big)\dot{+}\mathcal D\big(\,\overline{b}\big)\, \Big)
\]
and it does not depend on~$E$.\smallskip

\noindent
Writing
\[\lambda_\infty:=\lim\lambda_k\in(\lambda_0,\infty]\]
one has
 \[
\lambda_\infty= \inf \big(\sigma_{\!\rm ess} (\widetilde A) \cap (\lambda_0, + \infty)\big)\;.
\]
In addition, the numbers $\lambda_k$ ($k\ge 1$) satisfying $\lambda_0<\lambda_k<\lambda_\infty$ are all the eigenvalues -- counted with multiplicity -- \, of $\widetilde A$ in the spectral gap $(\lambda_0,\lambda_\infty)$.
\end{theorem}
\noindent Theorem~\ref{Theorem:thm1} deserves some comments.
\\[2pt] $\bullet$
In some situations, one encounters a symmetric operator $A$ that does not satisfy \eqref{H1} but has the weaker property $\Lambda_\pm \mathcal D(A)\subset \mathcal D(\overline{A})$, where $\overline{A}$ denotes the closure of $A$. This happens for instance if one defines a Dirac-Coulomb operator on a ``minimal" domain consisting of compactly supported smooth functions, and one considers the splitting associated with the free energy projectors $\Lambda_\pm=\mathbbm 1_{\,\R_{\pm}}(D)\,$: see the example of Subsection \ref{subsection: sign-changing}. In such a case one can replace $A$ by its symmetric extension $\overline{A}\upharpoonright_{_{\Lambda_+ \mathcal D(A)\oplus\Lambda_-\mathcal D(A)}}$. Then \eqref{H1} is automatically satisfied by the new domain and it remains to check that the new operator satisfies \eqref{H2}-\eqref{H3} before applying Theorem~\ref{Theorem:thm1}.
\\[2pt] $\bullet$
In the earlier works~\cite{GS99,Griesemer1999,DES00a,DES00b,DES03,MorMul-15,M16,ELS19,SST20} on the min-max principle in gaps, one assumes that $\lambda_1>\lambda_0\,,$ which amounts to consider assumption~\eqref{H3} with $k_0=1$. Allowing $k_0\ge 2$ can be important in some applications: see Section~\ref{Section:DC}. The abstract min-max principle for eigenvalues in the case $k_0\ge 2$ was first considered in~\cite{DES06}, but in that paper~\eqref{H2} was replaced by a much more restrictive assumption. Moreover, the proof in~\cite{DES06} was based on the arguments of~\cite{DES00a}, so it suffered from the same closability issue solved in~\cite{SST20} and~the corrigendum~\cite{DES2023corrigendum} of~\cite{DES00a}: the closure of $L_E$ was used but its existence was not proved.
\\[2pt] $\bullet$ Compared with~\cite{EstLos-08,SST20}, another important novelty is that for constructing $\widetilde A$ we do not need the operator $\,-\Lambda_- A\upharpoonright_{_{F_-}}\,$ to be self-adjoint or essentially self-adjoint in $\mathcal H_-\,$. This assumption was used in~\cite{SST20} to prove that $L_E$ is closable, while in the present work this closability is not needed thanks to a new geometric viewpoint: instead of trying to close $L_E$ we consider the subspace $\overline{\Gamma}_E$, which of course always exists, but is not necessarily a graph. As pointed out in~\cite{SST20}, essential self-adjointness of $\,-\Lambda_- A\upharpoonright_{_{F_-}}\,$ holds true in many important situations. However there are also interesting examples for which it does not hold true. An application to Dirac-Coulomb operators in which the essential self-adjointness of $\,-\Lambda_- A\upharpoonright_{_{F_-}}\,$ does not hold true is described in Subsection~\ref{subsection: sign-changing}. Let us give a simpler example: on the domain $F:=\(C^\infty_c(\Omega, \R)\)^2 $ consider the operator
\be{exampleDelta}
A\begin{pmatrix} u\\v\end{pmatrix}:= \begin{pmatrix} -\Delta u\\\Delta v\end{pmatrix}
\ee
taking values in $\mathcal H= \(L^2(\Omega, \R)\)^2$, where $\Omega$ is a bounded open subset of $\,\R^d$ with smooth boundary. In this case one takes 
\[
\Lambda_+\begin{pmatrix} u\\v\end{pmatrix}=\begin{pmatrix} u\\0\end{pmatrix}\,,\quad\Lambda_-\begin{pmatrix} u\\v\end{pmatrix}=\begin{pmatrix} 0\\v\end{pmatrix}
\]
and~\eqref{H1} holds true.
If $\lambda(\Omega)>0$ is the smallest eigenvalue of the Dirichlet Laplacian on $\Omega$, we have $\lambda_0=-\lambda(\Omega)$ in~\eqref{H2} and $\lambda_1= \lambda(\Omega)>\lambda_0$, so~\eqref{H3} with $k_0=1$  holds true. But $-\Lambda_- A\upharpoonright_{_{F_-}}$ is the Laplacian defined on the minimal domain $\{0\}\times C^\infty_c(\Omega, \R)\,$, and it is well-known that this operator is not essentially self-adjoint in $\{0\}\times L^2(\Omega, \R)\,$, so one cannot apply the abstract results of~\cite{EstLos-08,SST20}.
In this example, the distinguished extension given by Theorem~\ref{Theorem:thm1} is easily obtained as follows. One checks that $\mathcal D\big(\,\overline{b}\big)= \{0\}\times H^1_0(\Omega, \R)\,$, $\mathcal D(\overline{q}_E)= H^1_0(\Omega, \R)\times \{0\}$ and $\mathcal D\big(\,A^*\big)= \left( H^2(\Omega, \R)\right)^2\,$. So, denoting by $\Delta^{(D)}$ the Dirichlet Laplacian with domain $H^2(\Omega, \R)\cap H^1_0(\Omega, \R)\,$, one finds that
\[\widetilde A=\begin{pmatrix} -\Delta^{(D)} & 0\\0 & \Delta^{(D)}\end{pmatrix}\,.\]
\\[2pt] $\bullet$ In~\cite{SST20}, it is proved that the extension $\widetilde A$ is unique among the self-adjoint extensions whose domain is included in $\Lambda_+\mathcal D(\overline{q}_E) \,\oplus\, \mathcal H_-$, assuming that the operator $-\Lambda_-A\upharpoonright_{_{F_-}}\,$ is essentially self-adjoint. But the above example shows that without this assumption, such a uniqueness result does not hold true in general. Indeed, since $\Delta: C^\infty_c(\Omega, \R) \rightarrow L^2(\Omega, \R)$ is not essentially self-adjoint, the operator $A$ given by \eqref{exampleDelta} has infinitely many self-adjoint extensions with domains included in $\Lambda_+\mathcal D(\overline{q}_E) \,\oplus\, \, \mathcal H_-$. For instance, one can take
\[\hat{A}=\begin{pmatrix} -\Delta^{(D)} & 0\\0 & \Delta^{(N)}\end{pmatrix}\]
with $\Delta^{(N)}$ the self-adjoint extension of $\Delta$ associated with the Neumann boundary condition $\nabla v\cdot n=0\,$, where $n$ denotes the outward normal unit vector on  $\partial\Omega$. Obviously, $\hat{A}\neq \widetilde A$.
\\[2pt] $\bullet$ As we will see in Subsection \ref{subsection: attractive}, when dealing with the Dirac-Coulomb operator $D_{-\nu/|x|}$ with Talman's splitting it is natural to choose $F=C^\infty_c(\R^3\setminus\{0\},\C^4)$. Then the large and small two-components appearing in Talman's min-max principle are taken in $C^\infty_c(\R^3\setminus\{0\},\C^2)$. But other functional spaces can be used for these components. In~\cite{MorMul-15,M16} an abstract min-max principle is stated in the setting of quadratic forms and applied to Talman's min-max principle with $H^{1/2}(\R^3,\C^2)$ as space of large and small two-components, under the condition $\nu\in\left(0,1\right)$. However it seems that some closability issues are present in the proof of the abstract principle, as in~\cite{DES00a}. We do not know whether the geometric approach of the present paper could be adapted to the framework of~\cite{MorMul-15} in order to avoid these closability issues without additional assumptions. Note that by a completely different approach, Talman's min-max principle is proved in \cite{ELS19} for all $\nu\in (0,1]$, with arbitrary spaces of large and small two-components lying between $C^\infty_c(\R^3\setminus\{0\},\C^2)$ and $H^{1/2}(\R^3,\C^2)$.\medskip

\noindent Concerning the proof of Theorem \ref{Theorem:thm1}, we emphasize three main facts:

\noindent
(1) {\it The quadratic form $q_E(x)=\innerproduct{x}{S_Ex}$ is bounded from below for all $E>\lambda_0$, so that it has a closure $\overline{q}_E$ in $\overline{\Gamma}_E$ and $S_E$ has a Friedrichs extension $T_E$}. This fact will allow us to define the distinguished extension $\widetilde A$ as the restriction of $A^*$ to $\mathcal D(A^*)\cap\(\mathcal D\big(\overline{q}_E\big)+\mathcal D\big(\overline{b}\big)\)$. We will prove its symmetry thanks to a formula expressing the product $\innerproduct{\(\widetilde A-E\)X}{U}$ in terms of $\overline{q}_E$ and $\overline{b}_E\,$, for $X\in \mathcal D(\widetilde A)$ and $U\in \mathcal D\big(\overline{q}_E\big)+\mathcal D\big(\overline{b}\big)$. This formula will be deduced by density arguments from a decomposition of $a-E\Vert\cdot\Vert^2\,$ as the difference of ${q}_E$ and $\overline{b}_E$.

\noindent
(2) {\it The self-adjoint operator $T_E$ is invertible for all $\lambda_0<E<\lambda_{k_0}$}. Combined with the (obvious) invertibility of $B+E$, this fact will allow us to construct the inverse of the distinguished extension $\widetilde A-E$, by using once again the formula relating $\innerproduct{\(\widetilde A-E\)X}{U}$ to $\overline{q}_E$ and $\overline{b}_E$. Then, by a classical argument, we will conclude that $\widetilde A$ is self-adjoint.

\noindent
(3) Although we are not able to prove that $\overline{\Gamma}_E$ is a graph above $\mathcal H_+$, we will see that the sum $\overline{\Gamma}_E+ \mathcal D\big(\,\overline{b}\big)$ is direct in the algebraic sense. More importantly, if we denote by $\pi_E:\,\overline{\Gamma}_E\dot{+} \mathcal D\big(\,\overline{b}\big)\to\overline{\Gamma}_E$ and $\,\pi'_E:\,\overline{\Gamma}_E\dot{+} \mathcal D\big(\,\overline{b}\big)\to\mathcal D\big(\,\overline{b}\big)$ the associated projectors, {\it the linear map
\[X\in\mathcal D(\widetilde A)\mapsto \big(\pi_E X,\pi'_E X\big)\in \overline{\Gamma}_E\times \mathcal D\big(\,\overline{b}\big)\]
is continuous for the norms $\Vert X\Vert_{\mathcal D(\widetilde A)}$ and $\Vert \pi_E X\Vert+\Vert \pi'_E X\Vert$}. Thanks to this fact, we will be able to give a relation between the spectra of $\widetilde A$ and $T_E$ which will allow us to prove the min-max principle for the eigenvalues of $\widetilde A$ above $\lambda_0$.

\noindent
For $k_0=1$ the facts (1) and (2) are a consequence of the positivity of $q_E$ for $E\in (\lambda_0,\lambda_{1})$ and of the Riesz isomorphism theorem. When $k_0\geq 2$ the positivity is lost, but these two key facts still hold true for other reasons to be given in the proofs of Proposition ~\ref{Prop:thm1-abis} and Lemma~\ref{bounded below} and in Remark~\ref{k0=1}.
\bigskip

\noindent The paper is organized as follows. In Section~\ref{Section:Gen}, we study the quadratic form $q_E$ under Assumptions~\eqref{H1}-\eqref{H2}. In Section~\ref{bounded-below}, under the additional condition~\eqref{H3} we prove that $q_E$ is bounded from below, then we study its closure $\overline{q}_E$ and the Friedrichs extension $T_E$. The self-adjoint extension $\widetilde A$ is constructed in Section~\ref{Section:mainproof} and the abstract version of Talman's principle for its eigenvalues is proved in Section~\ref{Section:min-max}, which ends the proof of Theorem~\ref{Theorem:thm1}. Section~\ref{Section:DC} is devoted to Dirac-Coulomb operators with charge configurations that are not covered by earlier abstract results.

\section{The quadratic form \texorpdfstring{$q_E$}{qE}}\label{Section:Gen}

The results of this section are essentially contained in the earlier works ~\cite{DES00a,EstLos-08,SST20}, we recall them here for the reader's convenience. We first give a more intuitive interpretation of the objects $L_E$, $q_E$ that have been defined in the Introduction. Then we define a sequence of min-max levels for $q_E$ that will be related  to the min-max levels $\lambda_k$ of $A$ in Section~\ref{Section:min-max}. 

\subsection{A family of maximization problems}\label{decomp}
 
In this subsection we motivate the definition of $L_E$ and $q_E$ given in the Introduction. Consider the eigenvalue equation $(A-E)x=0$ with unknowns $x\in F$ and $E\in \mathbb R$. Writing $x_+:=\Lambda_+x\,$, $y_-:=\Lambda_-x$ and projecting both sides of the equation on $\mathcal H_-$,
 one gets
\[\Lambda_-Ax_++\Lambda_-(A-E)y_-=0\,\]
which is also the Euler-Lagrange equation for the problem
\[\sup_{y_-\in F_-} \innerproduct{x_++y_-}{(A-E)\,(x_++y_-)}\,.\]
Given $x_+$ in $F_+$ one can try to look for a solution $y_-$. In general the problem is not solvable in $F_-$ but one can consider a larger space in which a solution exists. We denote by $L_Ex_+$ the generalized solution. In order to make these explanations more precise, we need to express the quadratic form $a-E\Vert \cdot\Vert^2\,$ in terms of $q_E$, $\overline{b}_E$.
 
 Given $x_+\in F_+$ and $E>\lambda_0$, let $\;\varphi_{E, x_+}: F_- \rightarrow \R$ be defined by
\[
\varphi_{E, x_+}(y_-):= \innerproduct{x_++y_-}{(A-E)\,(x_++y_-)} \,,\quad\forall\,y_-\in F_-\,,
\]
One easily sees that $\varphi_{E, x_+}$ has a unique continuous extension to $\mathcal D\big(\overline{b}\big)$ which is the strictly concave function
\[
\overline{\varphi}_{E, x_+}:y_-\in \mathcal D\big(\overline{b}\big)\mapsto \innerproduct{x_+}{(A-E)\,(x_+)}+2\Re\innerproduct{Ax_+}{y_-} -\overline{b}_E(y_-) \,.
\]

The main result of this subsection is
\begin{proposition}\label{Prop:thm1-a}
Let $A$ be a symmetric operator on the Hilbert space $\mathcal H$. Assume hypotheses~\eqref{H1}-\eqref{H2}, take $E>\lambda_0$ and remember the definition~\eqref{def Q} of $q_E$. Then:

\noindent
$\bullet$ One has the decomposition
\be{concave}
\innerproduct{X}{(A-E)X}= q_E(\Lambda_+X+L_E\,\Lambda_+X)- \overline{b}_E(\Lambda_-X-L_E\,\Lambda_+X)\;\, ,\quad\forall X\in F\,.
\ee
$\bullet$ For each $x_+\in F_+$, $L_E x_+$ is the unique maximizer of  $\;\,\overline{\varphi}_{E, x_+}$ and one has
\be{sup}
q_E(x_++L_Ex_+)=\overline{\varphi}_{E, x_+}(L_Ex_+)=\max_{y_-\in F_-} \overline{\varphi}_{E, x_+}(y_-)=\sup_{y_-\in F_-} \varphi_{E, x_+}(y_-)\;.
\ee
\end{proposition}
\begin{proof} If $X\in F$, taking $x_+:=\Lambda_+X\in F_+$, $y_-:=\Lambda_-X\in F_-$ and $z_-:= y_--L_E\,x_+\in \mathcal D(B)$ we obtain
\begin{align*}
\innerproduct{X}{(A-E)X}&= \innerproduct{x_+}{(A-E)\,x_+}+2\,\Re\innerproduct{A\,x_+}{y_-}-\innerproduct{y_-}{(B+E)\,y_-}\\
&= \innerproduct{x_+}{(A-E)\,x_+}+2\,\Re\innerproduct{L_E\,x_+}{(B+E)\,y_-}-\innerproduct{y_-}{(B+E)\,y_-}\\
&= \innerproduct{x_+}{(A-E)\,x_+}+\innerproduct{L_E\,x_+}{(B+E)\,L_E\,x_+}\\
&\hspace*{4cm}+\Re \innerproduct{L_E\,x_+}{(B+E)\,z_-} -\Re \innerproduct{z_-}{(B+E)\,y_-}\\
&= \innerproduct{x_+}{(A-E)\,x_+}+\innerproduct{L_E\,x_+}{(B+E)\,L_E\,x_+}-\innerproduct{z_-}{(B+E)\,z_-}\,,
\end{align*}
which proves~\eqref{concave}. Now, given $x_+\in F_+$ this identity can be rewritten
in the form
\[ \varphi_{E, x_+}(y_-)=q_E(x_++L_E\,x_+)- \overline{b}_E(y_--L_E\,x_+)\,,\quad\forall y_-\in F_-\,.\]
By density of $F_-$ in the Hilbert space $\(\mathcal D(\overline{b}),\overline{b}_E(\cdot,\cdot)\)$ one thus has
\[ \overline{\varphi}_{E, x_+}(y_-)=q_E(x_++L_E\,x_+)- \overline{b}_E(y_--L_E\,x_+)\,,\quad\forall y_-\in \mathcal D\big(\overline{b}\big)\]
and by the positivity of $\overline{b}_E$ one concludes that~\eqref{sup} holds true,
which completes the proof.\end{proof}

\subsection{The min-max levels for \texorpdfstring{$q_E$}{qE}}\label{dependence}

If assumptions~\eqref{H1} and~\eqref{H2} hold true, to each $E>\lambda_0$ we may associate the nondecreasing sequence of min-max levels $\Big(\ell_k(E)\Big)_{k\geq 1}$ defined by
\be{versionreguliere}
\ell_k(E)\;:=\;\inf_{\scriptstyle\begin{array}{c}\mbox{\scriptsize $V$ subspace of $\Gamma_E$}\\[-4pt] \mbox{\scriptsize dim$\,V=k$}\end{array}}\,\sup_{x\in V\setminus\{0\}}\;\frac{ q_E(x)}{\|x\|^2}\;\;\in\; [-\infty,+\infty)\,.
\ee
We may also define the (possibly infinite) multiplicity numbers
\be{mult}
m_k (E) \; := \;{\rm card} \bigl\{ k' \ge 1 \,:\, \,\ell_{k'} (E) = \ell_k (E) \bigr\} \,\ge \,1\,.
\ee
In this subsection we analyse the dependence on $E$ of the quadratic form $q_E$ and its associated min-max levels. The results are summarized in the following proposition:
\begin{proposition}\label{Prop:QE-comparison}
Assume that {\rm~\eqref{H1}-\eqref{H2}} of Theorem~\ref{Theorem:thm1} are satisfied. Then:

$\bullet$ For all $\lambda_0<E<E'$ and for all $x_+\in F_+$, we have
\be{ineqq1}
\big\|x_++L_{E'}x_+\big\|\le\big\|x_++L_E\,x_+\big\|\le\frac{E'-\lambda_0}{E-\lambda_0}\,\big\|x_++L_{E'}x_+\big\|
\ee
and
\be{ineqq2}
(E'-E)\,\big\|x_++L_{E'}x_+\big\|^2\le q_E\big(x_++L_E\,x_+\big)-q_{E'}\big(x_++L_{E'}x_+\big)\le (E'-E)\,\big\|x_++L_E\,x_+\big\|^2\,.
\ee

$\bullet$ For every positive integer $k$ and all $\lambda>\lambda_0$, one has
\be{above}\ell_k(\lambda)\leq \lambda_k-\lambda\,.\ee

$\bullet$ For every positive integer $k$, if $\lambda_k>\lambda_0\,$ then for all $\lambda>\lambda_0$, one has
\be{below}\ell_k(\lambda)\geq (\lambda_k-\lambda)\left(\frac{\lambda-\lambda_0}{\lambda_k-\lambda_0}\right)^2\,.\ee
As a consequence, when $\lambda_k>\lambda_0\,$ the min-max level $\,\ell_k(\lambda)$ is finite for every $\lambda>\lambda_0$. It is positive when $\lambda_0<\lambda<\lambda_k$, negative when $\lambda>\lambda_k$ and one has $\ell_k ( \lambda) \,= \,0$ if and only if $\lambda=\lambda_k$. Therefore, the formula
$m_k(\lambda_k)=\,{\rm card} \bigl\{ k'
\ge 1 \,:\, \,\lambda_{k'} = \lambda_k \bigr\} \,$ holds true.
\end{proposition}
\begin{proof} Both formula \eqref{ineqq1} and \eqref{ineqq2} as well as their detailed proof can be found in~\cite[Lemma 2.1]{DES00a} and \cite[Lemma~2.4]{SST20}, so here we just  give the main arguments. In order to prove \eqref{ineqq1} one can start from the fact that for all $t\geq-\lambda_0$, $(t+E')^{-1}\leq (t+E)^{-1}\leq \frac{E'-\lambda_0}{E-\lambda_0}(t+E')^{-1}$. Then one can use the inclusion $\sigma(B)\subset [-\lambda_0,\infty)$ and the definition of $L_E$. In order to prove $\eqref{ineqq2}$, one notices that this formula is equivalent to the two inequalities $\overline{\varphi}_E(L_{E'}x_+)\leq q_E(x_++L_Ex_+)$ and $\overline{\varphi}_{E'}(L_{E}x_+)\leq q_{E'}(x_++L_{E'}x_+)$, which both hold true thanks to \eqref{sup}.

\noindent We now prove \eqref{above} and \eqref{below}.
 
By definition of $\lambda_k$, for each $\varepsilon>0$ there is a $k$-dimensional subspace $V_\varepsilon$ of $F_+$ such that for all $x_+\in V_\varepsilon$ and $y_-\in F_-$, $a(X)\leq (\lambda_k+\varepsilon) \Vert X\Vert^2$ with $X=x_++y_-$. If $E\in (\lambda_0,\infty)$ this inequality can be rewritten as $\varphi_{E,x_+}(y_-)\leq (\lambda_k-E+\varepsilon)\Vert x_++y_-\Vert^2$. By a density argument one infers that the inequality
$$\overline{\varphi}_{E,x_+}(y_-)\leq (\lambda_k-E+\varepsilon)\Vert x_++y_-\Vert^2$$
holds true for all $y_-\in\mathcal D(\overline{b})$. Choosing $y_-=L_Ex_+$ and using \eqref{sup}, one gets the estimate $q_E(x)\leq (\lambda_k-E+\varepsilon)\Vert x\Vert^2$ with $x=x_++L_Ex_+$, hence
\[\sup_{x\in W_\varepsilon\setminus\{0\}}\frac{q_E(x)}{\Vert x\Vert^2}\leq \lambda_k-E+\varepsilon\]
with $W_\varepsilon:=\{x\in \Gamma_E\,:\,\Lambda_+x\in V_\varepsilon\}\,.$ Since $\varepsilon$ is arbitrary and ${\rm dim}(W_\varepsilon)=k$, we conclude that
\eqref{above} holds true.

On the other hand, using once again the definition of $\lambda_k$, we find that for each $\varepsilon>0$ and each $k$-dimensional subspace $W$ of $\Gamma_{\lambda_k}$, there is a nonzero vector $x_\varepsilon$ in the $k$-dimensional space $V:=\Lambda_+W\subset F_+$ and a vector $y_\varepsilon\in F_-$ such that $a(X_\varepsilon)\geq (\lambda_k-\varepsilon) \Vert X_\varepsilon\Vert^2$ with $X_\varepsilon=x_\varepsilon+y_\varepsilon$. If $\lambda_k>\lambda_0$, after imposing $\varepsilon<\lambda_k-\lambda_0$ we get
${\varphi}_{\lambda_k-\varepsilon,x_\varepsilon}(y_\varepsilon)\geq 0\,$, hence, invoking \eqref{sup},
$q_{\lambda_k-\varepsilon}(x_\varepsilon+L_{\lambda_k-\varepsilon}x_\varepsilon)\geq 0$. Then, using \eqref{ineqq1}, \eqref{ineqq2} with the choices $E=\lambda_k-\varepsilon$, $E'=\lambda_k$, we get
\[q_{\lambda_k}(x_\varepsilon+L_{\lambda_k}x_\varepsilon)\geq q_{\lambda_k-\varepsilon}(x_\varepsilon+L_{\lambda_k-\varepsilon}x_\varepsilon)-\varepsilon\Vert x_\varepsilon+L_{\lambda_k-\varepsilon}x_\varepsilon\Vert^2\geq -\varepsilon\left(\frac{\lambda_k-\lambda_0}{\lambda_k-\varepsilon-\lambda_0}\right)^2\Vert x_\varepsilon+L_{\lambda_k}x_\varepsilon\Vert^2\,.\]
Since $W$ and $\varepsilon$ are arbitrary, we thus have $\ell_k(\lambda_k)\geq 0\,.$ Combining this with \eqref{above}, we see that $\ell_k(\lambda_k)= 0$.

It remains to study the case $\lambda_k>\lambda_0$ and $\lambda\in(\lambda_0,\infty)\setminus\{\lambda_k\}$. We take an arbitrary $k$-dimensional subspace $\widehat{W}$ of $\Gamma_\lambda$. We define $V:=\Lambda_+\widehat{W}\subset F_+$ and $W:=\{x=x_++L_{\lambda_k}x_+:\,x_+\in V\}\subset\Gamma_{\lambda_k}$. Then $W$ is also $k$-dimensional, so one has $\sup_{x\in W\setminus\{0\}}\frac{q_{\lambda_k}(x)}{\Vert x\Vert^2}\geq 0$, from what we have just seen. So, by compactness of the unit sphere for $\Vert\cdot\Vert$ of the $k$-dimensional space $W$ and the continuity of $q_{\lambda_k}$ on this space, there is $x_0\in V$ such that $\Vert x_0+L_{\lambda_k}x_0\Vert=1$ and
$q_{\lambda_k}(x_0+L_{\lambda_k}x_0)\geq 0$. In order to bound $q_{\lambda}(x_0+L_{\lambda}x_0)$ from below, we use \eqref{ineqq1}, \eqref{ineqq2} with $E=\min(\lambda,\lambda_k)$ and $E'=\max(\lambda,\lambda_k)$. We get
\[q_\lambda(x_0+L_\lambda x_0)\geq (\lambda_k-\lambda)\Vert x_0+L_{\lambda_k}x_0\Vert^2\geq (\lambda_k-\lambda)\left(\frac{\lambda-\lambda_0}{\lambda_k-\lambda_0}\right)^2\Vert x_0+L_{\lambda}x_0\Vert^2\,.\]

Since $\widehat{W}$ is arbitrary, we conclude that \eqref{below} holds true.

The last statements of Proposition \ref{Prop:QE-comparison} - finiteness and sign of $\ell_k(\lambda)$, the fact that $\lambda_k$ is the unique zero of $\ell_k$ - are an immediate consequence of \eqref{above} and \eqref{below}. Note that this characterization of $\lambda_k$ as unique solution of a nonlinear equation was already stated and proved in~\cite[Lemma~2.2~(c)]{DES00a} and~\cite[Lemma~2.8~(iii)]{SST20}.
\end{proof}

\begin{remark}\label{check H3} Assumptions \eqref{H1}-\eqref{H2} are rather easy to check in practice, but checking \eqref{H3} is more delicate. The second point in Proposition \ref{Prop:QE-comparison} provides a way to do this: one just has to prove that for some $k_0\geq 1$ and $E_0>\lambda_0$ the level $\ell_{k_0}(E_0)$ is nonnegative, which implies that $\lambda_{k_0}\geq E_0$. In Section \ref{Section:DC} we will apply this method to one-center and multi-center Dirac-Coulomb operators.
\end{remark}
\begin{remark}\label{iteration} The numerical calculation of eigenvalues in a spectral gap is a delicate issue, due to a well-known phenomenon called {\it spectral pollution}: as the discretization is refined, one sometimes observes more and more {\it spurious} eigenvalues that do not approximate any eigenvalue of the exact operator (see \cite{LewSer-10}). It is possible to eliminate these spurious eigenvalues thanks to Talman's min-max principle. A method inspired of Talman's work was proposed in~\cite{DolEstSerVan-00,DES03}. The idea was to calculate each eigenvalue $\lambda_k$ as the unique solution of the problem $\ell_k(\lambda)=0$. This method is free of spectral pollution, but solving nonlinear equations has a computational cost. The estimates \eqref{above} and \eqref{below} proved in the present work suggest a fast and stable iterative algorithm that could reduce this cost. Starting from a value $E^{(0)}$ comprised between $\lambda_0$ and $\lambda_k$, one can compute a sequence of approximations by the formula $E^{(j+1)}=E^{(j)}+\ell_k\left(E^{(j)}\right)$. From \eqref{above}, one proves by induction that for all $j\geq 0$, one has $E^{(j)}\in(\lambda_0,\lambda_k)$,   $\, E^{(j+1)}-E^{(j)}=\ell_k\left(E^{(j)}\right)>0$ and $E^{(j)}$ converges monotonically to $\lambda_k$. Moreover, combining the inequalities \eqref{above} and \eqref{below} one finds that for $\vert h \vert$ small, $h+\ell_k(\lambda_k+h)=\mathcal O(h^2)$. So $E^{(j)}$ converges quadratically to $\lambda_k$. It would be interesting to perform numerical tests of this algorithm in practical situations.
\end{remark}

\section{The closure \texorpdfstring{$\overline{q}_E$}{qE} and the Friedrichs extension \texorpdfstring{$T_E$}{TE}}\label{bounded-below}

 In this section, under assumptions~\eqref{H1}-\eqref{H2}-\eqref{H3} we prove that the form $q_E$ is boun\-ded from below and closable, so that the Schur complement $S_E$ has a Friedrichs extension $T_E$. We then relate the spectrum of $T_E$ to the min-max levels $\lambda_k$. Finally, we construct a natural isomorphism between the domains of $\overline{q}_E$ and $\overline{q}_{E'}$ for all $E,\,E'\,>\lambda_0$.
 
 \subsection{Construction of \texorpdfstring{$\overline{q}_E$}{qE} and \texorpdfstring{$T_E$}{TE}}\label{bounded-from-below}
 
 In this subsection we are going to prove the following result:
 \begin{proposition}\label{Prop:thm1-abis} Let $A$ be a symmetric operator on the Hilbert space $\mathcal H$. Assuming~\eqref{H1}-\eqref{H2}-\eqref{H3} and with the above notations:
 
 $\bullet$ For each $E>\lambda_0$, the quadratic form~$q_E(x)=\innerproduct{x}{S_E x}$ is bounded from below hence closable in $\overline \Gamma_E$ and $S_E$ has a Friedrichs extension $T_E$.
 
 $\bullet$ If $E\in (\lambda_0,\lambda_\infty)\setminus\{\lambda_k\,:\, k\geq k_0\}$ then $T_E:\mathcal D(T_E)\to \overline{\Gamma}_E$ is invertible with bounded inverse. If $\lambda_0<\lambda_k<\lambda_\infty$ then $0$ is the $k$-th eigenvalue of $T_{\lambda_k}$ counted with multiplicity. Moreover its multiplicity is finite and equal to $\,{\rm card} \bigl\{ k'
\ge 1 \,:\, \,\lambda_{k'} = \lambda_k \bigr\} \,$.
If $\,\lambda_k=\lambda_\infty$ for some positive integer $k$, then $\,0=\min\,\sigma_{\!\rm ess}(T_{\lambda_k})\,$.\end{proposition}

\noindent The main tool in the proof of Proposition \ref{Prop:thm1-abis} is the following result:
\begin{lemma}\label{bounded below}
Under assumptions~\eqref{H1}-\eqref{H2}-\eqref{H3}, for every $E>\lambda_0\,$, there is $\kappa_E>0$ such that $q_E+\kappa_E\,\|\cdot\|^2 \ge \|\cdot\|^2$ on $\Gamma_E\,$.
\end{lemma}
\begin{proof}
We distinguish two cases depending on the value of $k_0=\min\{k\geq 1\,:\,\lambda_{k}>\lambda_0\}\,$.

When $k_0=1$, one has $\lambda_1>\lambda_0$ and $q_{\lambda_1}(x)\geq 0$ for all $x\in \Gamma_{\lambda_1}$. So, using the inequalities~\eqref{ineqq1} and~\eqref{ineqq2}, one finds that 
for all $E>\lambda_0$ and $x\in\Gamma_E$,
$q_E(x)+\kappa_E\Vert x\Vert^2\geq \Vert x\Vert^2\,$,
with
\[
\kappa_E:=1+\max\big\{0, (E-\lambda_1)\big\}\left(\frac{E-\lambda_0}{\lambda_1-\lambda_0}\right)^2\,.
\]

When $k_0\geq 2$ we need a different argument and the formula for $\kappa_E$ is less explicit. As in the case $k_0=1\,$, we just have to find a constant $\kappa_E$ for {\it some} $E>\lambda_0\,$; then the inequalities~\eqref{ineqq1} and~\eqref{ineqq2} will immediately imply its existence for {\it all} $E>\lambda_0\,$. We take $E\in (\lambda_0,\lambda_{k_0})\,$. Since $\lambda_{k_0-1}=\lambda_0<E\,$, by the second point of Proposition \ref{Prop:QE-comparison} we have $\ell_{k_0-1}(E)\in [-\infty,0)\,$. So there is a $(k_0-1)$-dimensional subspace $W$ of $\Gamma_E$ such that
$$\ell':=\sup_{w\in W\setminus \{0\}} \frac{q_E(w)}{\|w\|^2} \in (-\infty,0)\,.$$
Let $C:=\sup\big\{\| S_E w\|\,:\,w\in W\,,\,\| w \| \le 1\big\}\,$. This constant is finite, since the space $W$ is finite-dimensional. We now consider an arbitrary vector $x$ in $\Gamma_E\,$ and we look for a lower bound on $q_E(x)\,$. We distinguish two cases.
\\[4pt]
\noindent$\bullet$ {\it First case:} $x\in W\,$. Then $q_E(x)=\innerproduct{x}{S_E x}\ge -\,C\, \|x\|^2$.
\\[4pt]
\noindent$\bullet$ {\it Second case:} $x\notin W\,$. Then the vector space $\mathrm{span}\{x\}\oplus W$ has dimension $k_0\,$. Since $\lambda_{k_0}>E>\lambda_0\,$, by the third point of Proposition \ref{Prop:QE-comparison} we obtain $\ell_{k_0}(E)>0\,$, so there is a vector $w_0\in W$ such that $q_E(x+w_0)\ge 0\,$. Then we have
\[
q_E(x)=q_E(x+w_0)-2\,\Re \innerproduct{x}{S_E w_0} -q_E(w_0)\ge-2\,C\,\|x \| \| w_0\| +|\ell'| \,\| w_0\|^2\ge -\frac{C^2}{|\ell'|}\,\|x \|^2\,.
\]
So in all cases, if we choose $\kappa_E=1+\max\big\{C,C^2/|\ell'|\big\}\,$, we get $q_E(x)+\kappa_E\,\|x \|^2\ge \|x \|^2\,$. This completes the proof of the lemma.
\end{proof}

\noindent{\it Proof of Proposition~\ref{Prop:thm1-abis}}. 
As mentioned in the Introduction, we have $q_E(x)=\innerproduct{x}{S_Ex}$ where $S_E:\, \Gamma_E\to \overline \Gamma_E$ is the Schur complement of the block decomposition of $A-E$ under the splitting $\mathcal H=\mathcal H_+\oplus\mathcal H_-\,$ given by formula \eqref{def S}.
 The operator $S_E$ is densely defined in the Hilbert space $\overline \Gamma_E$ and it is clearly symmetric, moreover we have just seen that $q_E$ is bounded from below, so $q_E$ is closable in $\overline \Gamma_E$. We denote its closure by $\overline q_E$. We call $T_E$ the Friedrichs extension of~$S_E$ in $\overline \Gamma_E$. With
\[\ell_\infty(E):=\lim_{k\to\infty} \ell_k(E)\,,\]
the classical min-max principle implies that the levels $\ell_k(E)$ such that $\ell_k(E)<\ell_\infty(E)$ are all the eigenvalues of $T_E$ below $\ell_\infty(E)$ counted with multiplicity, and one has $\ell_\infty(E)=\inf \sigma_{\!\rm ess}(T_E)$. So we have the following cases:

If $E\in (\lambda_0,\lambda_\infty)\setminus\{\lambda_k\,:\, k\geq 1\}$ then by Proposition \ref{Prop:QE-comparison},  one has $0<\ell_\infty(E)$ and $0$ is not in the set $\{\ell_k(E)\,:\, k\geq 1\}$. As a consequence, it is not in the spectrum of $T_E$, so $T_E$ is invertible with bounded inverse.

If $E=\lambda_k$ with $\lambda_0<\lambda_k<\lambda_\infty$ then, using once again Proposition \ref{Prop:QE-comparison}, we find that $\ell_k(E)=0$ and $\ell_\infty(E)>0$. So $0$ is an eigenvalue of $T_E$ of finite multiplicity equal to $m_k(\lambda_k)$ where $m_k$ has been defined in \eqref{mult}. From Proposition \ref{Prop:QE-comparison}, $m_k(\lambda_k)$ equals ${\rm card} \bigl\{ k'
\ge 1 \,:\, \,\lambda_{k'} = \lambda_k \bigr\} \,$.

If $\lambda_k=\lambda_\infty$, then for all $k'\geq k$ one has $\lambda_{k'}=\lambda_k$, so $\,0=\ell_{k'}(\lambda_k)$ by Proposition \ref{Prop:QE-comparison}. In other words, $0=\ell_\infty(\lambda_k)$. Then the classical min-max theorem implies that $0=\min\sigma_{\!\rm ess}(T_{\lambda_k})$.

\noindent Proposition \ref{Prop:thm1-abis} is thus proved.
\finprf

\begin{remark}\label{k0=1} When $k_0=1$ and $\lambda_0<E<\lambda_1$, the closed quadratic form $\overline{q}_E$ is positive definite and the invertibility of $T_E$ is just a consequence of the Riesz isomorphism theorem.\end{remark}

\subsection{A family of isomorphisms}\label{isom}

In the earlier works~\cite{DES00a} (before its corrigendum~\cite{DES2023corrigendum}), \cite{EstLos-07,EstLos-08} and~\cite{SST20}, $q_E$ was seen as a quadratic form on $F_+$ and the domain of its closure was independent of $E$. Note, however, that the existence of the closure was claimed without proof in~\cite{DES00a} and this was a serious gap. Moreover the proof of the domain invariance was based on an incorrect estimate in~\cite[Proposition 2]{EstLos-07} and was incomplete in~\cite{EstLos-08}, but this domain invariance problem is easily fixed using the inequalities (10)-(11) of \cite{DES00a} which are called \eqref{ineqq1}-\eqref{ineqq2} in the present work.

In our situation, since we do not know whether $L_E$ is closable or not, it is essential to define $q_E$ on $\Gamma_E\,$ and then to close it in~$\overline \Gamma_E$.
So the domain of ${q}_E$ cannot be independent of $E$: indeed it is a subspace of $\,\overline{\Gamma}_E\,$ which itself depends on $E$. However, if we endow each space $\mathcal D\big(\overline{q}_E\big)$ with the norm $\Vert x\Vert_{\mathcal D(\small{\overline{q}_E})}:=\sqrt{\overline{q}_E(x)+\kappa_E\Vert x\Vert^2}\,$,
there is a natural isomorphism $\hat{i}_{E,E'}$ between any two Banach spaces $\mathcal D\big(\overline{q}_E\big)$ and $\mathcal D\big(\overline{q}_{E'}\big)$, as explained in the next proposition:

\begin{proposition}
\label{Prop:comparison}
Under conditions~\eqref{H1}-\eqref{H2}-\eqref{H3}, for $E$, $E'\in (\lambda_0,\infty)\,,$ the linear map $i_{E,E'}:\,x_++L_E\,x_+\mapsto x_++L_{E'}x_+$ can be uniquely extended to an isomorphism $\hat{i}_{E,E'}$ between the Banach spaces $\mathcal D\big(\overline{q}_E\big)$ and $\mathcal D\big(\overline{q}_{E'}\big)$ which can itself be uniquely extended to an isomorphism $\overline{i}_{E,E'}$ between $\overline \Gamma_E$ and $\overline \Gamma_{E'}$ for the norm $\Vert\cdot\Vert$.
Moreover one has the formula
\be{extended} \overline{i}_{E,E'} (x)=x+(E-E')(B+E')^{-1}\Lambda_- x\;,\quad\forall x\in \overline{\Gamma}_E\,.
\ee
\end{proposition}
\begin{proof}
The linear map $i_{E,E'}$ is obviously a bijection between $\Gamma_{E}$ and $\Gamma_{E'}$, of inverse $i_{E',E}$. The estimates \eqref{ineqq1}, \eqref{ineqq2} of Proposition \ref{Prop:QE-comparison} imply that $i_{E,E'}$ is an isomorphism for the norm $\Vert\cdot\Vert$ as well as for the norms $\Vert \cdot\Vert_{\mathcal D(\small{\overline{q}_E})}$ and $\Vert \cdot\Vert_{\mathcal D(\small{\overline{q}_{E'}})}$, hence the existence and uniqueness of the successive continuous extensions $\hat{i}_{E,E'}$ and $\overline{i}_{E,E'}$, that are isomorphisms of inverses $\hat{i}_{E',E}$ and $\overline{i}_{E',E}$.

\noindent From the formula $L_E x_+=(B+E)^{-1}\Lambda_-Ax_+$ one easily gets the resolvent identity
$$ L_{E'}x_+=L_E x_+ + (E-E')(B+E')^{-1}L_E x_+\;,\quad\forall x_+\in F_+$$
which is the same as
$$ \overline{i}_{E,E'} (x)=x+(E-E')(B+E')^{-1}\Lambda_- x\;,\quad\forall x\in \Gamma_E\,.$$
By continuity of $\overline{i}_{E,E'}\,$, $(B+E')^{-1}$ and $\Lambda_-$ for the norm $\Vert\cdot\Vert\,$, one can extend the above formula to $\overline{\Gamma}_E$ and this ends the proof of the proposition.
\end{proof}

\noindent Proposition \ref{Prop:comparison} has the following consequence which will be useful in the next section:
\begin{corollary}\label{invariance}
Assume that conditions~\eqref{H1}-\eqref{H2}-\eqref{H3} hold true. Let $E$, $E'>\lambda_0$. Then
\be{form6}
\mathcal D(\overline{q}_{E'})\,+\, \mathcal D\big(\,\overline{b}\big) = \mathcal D(\overline{q}_E)\,+\, \mathcal D\big(\,\overline{b}\big)\,.
\ee
\end{corollary}
\begin{proof}

From formula \eqref{extended}, for each $x\in \mathcal D(\overline{q}_{E})$ one has $\hat{i}_{E,E'}\big(x) - x \in \mathcal D(B)\,$, hence
\[\mathcal D(\overline{q}_{E'})\subset
\mathcal D(\overline{q}_{E})+\mathcal D(B) \subset \mathcal D(\overline{q}_{E})+\mathcal D(\overline{b})\,,\]
which of course implies the inclusion
\[\mathcal D(\overline{q}_{E'})\,+\, \mathcal D\big(\,\overline{b}\big) \subset \mathcal D(\overline{q}_E)\,+\, \mathcal D\big(\,\overline{b}\big)\,.\]
Exchanging $E$ and $E'$ one gets the reverse inclusion, so~\eqref{form6} is proved.
\end{proof}

\section{The distinguished self-adjoint extension}\label{Section:mainproof}

In this section, we continue with the proof of Theorem~\ref{Theorem:thm1} by constructing the distinguished self-adjoint extension $\widetilde A$ and studying some properties of its domain that will be useful in the sequel. But before doing this, we need to establish a decomposition of the product $\innerproduct{(A^*-E)X}{U}$ under weak assumptions on the vectors $X$, $U$.

\subsection{A useful identity}\label{identity}

In this subsection we state and prove an identity that plays a crucial role in the construction and study of $\widetilde A$:
\begin{proposition}\label{Prop:formulaoperator}
Assume that conditions~\eqref{H1}-\eqref{H2}-\eqref{H3} hold true. Let $x$, $u\in \mathcal D(\overline{q}_E)$ and $z_-$, $v_-\in \mathcal D\big(\,\overline{b}\big)$ be such that $X=x+z_-\in \mathcal D\(A^*\)$. Then, with $U=u+v_-$, we have
\be{egalite-operator}
\innerproduct{(A^*-E)X}{U}= \overline q_E(x,u)-\overline b_E(z_-, v_-)
\ee
for every $E>\lambda_0$.
\end{proposition}
\begin{proof}
Formula~\eqref{concave} of Proposition \ref{Prop:thm1-a} exactly says that for all $X=x+z_-\in F$ with $x=\Lambda_+ X + L_E\Lambda_+ X$ and $z_-=\Lambda_-X-L_E\Lambda_+ X$, one has
$$\innerproduct{X}{(A-E)X}= \innerproduct{x}{S_E x}-\innerproduct{z_-}{(B+E)z_-}\,.$$
This relation between quadratic forms directly implies a formula involving their polar forms: for all $X=x+z_-\in F$ and $U=u+v_-\in F$ with $x=\Lambda_+ X + L_E\Lambda_+ X$, $u=\Lambda_+ U + L_E\Lambda_+ U$, $z_-=\Lambda_-X-L_E\Lambda_+ X$ and $v_-=\Lambda_-U-L_E\Lambda_+ U$, one has
\be{FF}\innerproduct{x+z_-}{(A-E)U}= \innerproduct{x}{S_E u}-\innerproduct{z_-}{(B+E)v_-}\,.\ee
In order to prove Proposition~\ref{Prop:formulaoperator}, we have to generalize~\eqref{FF} to larger classes of vectors $X,\,U$. We proceed in two steps.\medskip

\noindent
{\bf First step:} we fix $U=u+v_-$ in $F$, with $u=\Lambda_+ U + L_E\Lambda_+ U$ and $v_-=\Lambda_-U-L_E\Lambda_+ U$.\smallskip
 
\noindent
If $x=0$ and $X=z_-\in F_-$, the identity~\eqref{FF} holds true and reduces to
\be{egalite-1}
\innerproduct{z_-}{(A-E)U}= -\innerproduct{z_-}{(B+E)v_-}\,.
\ee
Both sides of~\eqref{egalite-1} are continuous in $z_-$ for the norm $\Vert \cdot\Vert$,
and we recall that $F_-$ is dense in $\mathcal H_-$. So~\eqref{egalite-1} remains true for all $z_-\in \mathcal H_-$.\smallskip

\noindent Now, in the special case $X=x_+\in F_+$, $x=x_++L_Ex_+$ and $z_-=-L_Ex_+$, the identity~\eqref{FF} becomes
\[ \innerproduct{x_+}{(A-E)U}= \innerproduct{x}{S_E u}+\innerproduct{L_Ex_+}{(B+E)v_-}\,.\]
We may also apply~\eqref{egalite-1} to $z_-=-L_E x_+$ and we get
\[ -\innerproduct{L_Ex_+}{(A-E)U}= \innerproduct{L_Ex_+}{(B+E)v_-}\,.\]
Subtracting these two identities, we find
\be{egalite-2}
\innerproduct{x}{(A-E)U}= \innerproduct{x}{S_E u}\,,\quad\forall x\in\Gamma_E\,.
\ee
Both sides of~\eqref{egalite-2} are continuous in $x$ for the norm $\Vert \cdot\Vert$, so~\eqref{egalite-2} remains true for all $x\in \overline{\Gamma}_E$.\smallskip

\noindent
Then, for $x\in \overline{\Gamma}_E$ and $z_-\in \mathcal H_-$, we may add~\eqref{egalite-1} and~\eqref{egalite-2}. We conclude that~\eqref{FF} remains valid for all $x\in \overline{\Gamma}_E$ and $z_-\in \mathcal H_-$. When $x\in \mathcal D(\overline{q}_E)$ and $z_-\in\mathcal D\big(\,\overline{b}\big)$, this identity may be rewritten in the form
\be{egalite-special}
\innerproduct{x+z_-}{(A-E)U} = \overline q_E(x,u)- \overline b_E(z_-,v_-)\,.
\ee
In particular, when $X=x+z_-\in \mathcal D(A^*)$ with $x\in \mathcal D(\overline{q}_E)$ and $z_-\in\mathcal D\big(\,\overline{b}\big)$, one obtains the equality
\be{egalite-0}
\innerproduct{(A^*-E)X}{U} = \overline q_E(x,u)- \overline b_E(z_-,v_-)\,.
\ee
This formula is the same as~\eqref{egalite-operator} under the additional assumptions $U\in F$, $u\in\Gamma_E$ and $v_-=U-u\in \mathcal D(B)$.\medskip

\noindent
{\bf Second step:} we fix $X=x+z_-\in \mathcal D(A^*)$ with $x\in \mathcal D(\overline{q}_E)$, $z_-\in\mathcal D\big(\,\overline{b}\big)\,$.\smallskip
\smallskip

\noindent
If $u=0$ and $v_-=U \in F_-$, the identity~\eqref{egalite-0} holds true and reduces to
\be{egalite-1bis}
\innerproduct{(A^*-E)X}{v_-}= -\,\overline b_E(z_-,v_-)\;.
\ee
Both sides of~\eqref{egalite-1bis} are continuous in $v_-$ for the norm $\Vert \cdot\Vert_{\mathcal D(\small{\overline{b}})}$
and we recall that $F_-$ is dense in $\mathcal D\big(\,\overline{b}\big)$ for this norm. So~\eqref{egalite-1bis} remains true for all $v_-\in \mathcal D\big(\,\overline{b}\big)$.\smallskip

\noindent Now, in the special case $U=u_+\in F_+$, $u=u_++L_E u_+$ and $v_-=-L_E u_+$, the identity~\eqref{egalite-0} becomes
\[ \innerproduct{(A^*-E)X}{u_+}= \overline q_E(x,u)+\overline b_E(z_-,L_E u_+)\,.\]
We may also apply~\eqref{egalite-1bis} to $v_-=-L_E u_+\in \mathcal D(B)\subset \mathcal D\big(\,\overline{b}\big)$ and we get
\[ -\innerproduct{(A^*-E)X}{L_E u_+}= \,\overline b_E(z_-,L_E u_+)\,.\]
Subtracting these two identities, we find
\be{egalite-2bis}
\innerproduct{(A^*-E)X}{u}= \overline q_E(x,u),\quad\forall u\in\Gamma_E\,.
\ee
Both sides of~\eqref{egalite-2bis} are continuous in $u$ for the norm $\Vert \cdot\Vert_{\mathcal D(\small{\overline{q}_E})}$, and we recall that $\Gamma_E$ is dense in $\mathcal D(\overline{q}_E)$ for this norm. So~\eqref{egalite-2bis} remains true for all $u\in \mathcal D(\overline{q}_E)$.\smallskip

\noindent
Then, for $u\in \mathcal D(\overline{q}_E)$ and $v_-\in \mathcal D\big(\,\overline{b}\big)$, we may add~\eqref{egalite-1bis} and~\eqref{egalite-2bis} and we finally get~\eqref{egalite-operator} in the general case.
\end{proof}

\subsection{Construction of the self-adjoint extension}\label{Section:sa-extension1}

In this subsection we prove
\begin{proposition}\label{Prop:thm1-b} Under conditions~\eqref{H1}~\eqref{H2}-\eqref{H3}, given $E>\lambda_0$ the operator $A$ admits a unique self-adjoint extension $\widetilde A$ such that $\mathcal D\big(\widetilde A\,\big)\subset \mathcal D\big(\,\overline{q}_E\big)+\mathcal D\big(\,\overline{b}\big)\,$. This extension is independent of $E$ and defined by
\be{def tilde A}
\widetilde A\,x\,:= A^*x\;,\quad \forall\,x\in \mathcal D\(\widetilde A\,\)
\ee
where
\be{domainE}
\mathcal D\big(\widetilde A\big):=\big(\mathcal D(\overline{q}_E)\,+\,\mathcal D\big(\,\overline{b}\big)\big)\cap\mathcal D\big(A^*\big)\,.
\ee

Moreover, for each $E$ in  $(\lambda_0,\lambda_\infty)\setminus\{\lambda_k\,:\,k_0\leq k<\infty\}\,$, the operator $\widetilde A-E$ is invertible with bounded inverse given by the formula
\be{inverse}
\left(\widetilde A-E\right)^{-1}=T_E^{-1}\circ \Pi_E -(B+E)^{-1}\circ \Lambda_-\,.
\ee
\end{proposition}

\begin{proof}\noindent For $E>\lambda_0$, the operator $\widetilde A$ defined by \eqref{def tilde A}-\eqref{domainE} is indeed an extension of $A$, since
\[
\mathcal D(A)= F\subset \big(\Gamma_E\,+\,\mathcal D\(B\)\big)\cap \mathcal D\big(A^*\big)\subset \mathcal D\big(\widetilde A\big)\quad\mbox{and $A^*\upharpoonright_{_{\mathcal D(A)}}\;=\,A$}\,.
\]
By Corollary \ref{invariance},
$\mathcal D\big(\widetilde A\big)$ is independent of $E$, as well as $\widetilde A=A^*\upharpoonright_{_{\mathcal D\(\widetilde A\,\)}}\,$. Moreover the extension $\widetilde A$ is symmetric: this immediately follows from Proposition \ref{Prop:formulaoperator}.

Now, given $f\in \mathcal H$ and $E>\lambda_0$ we want to study the equation $(\widetilde A-E)X=f\,$.
For this purpose, we introduce the following problem written in weak form:\\[4pt]
{\it \centerline{Find $(x, z_-) \in \mathcal D(\overline{q}_E)\,\times\, \mathcal D\big(\,\overline{b}\big)$ such that}}
\be{Pf}\tag{$\mathcal P_f$}
\left\{\begin{array}{ll}
\overline q_E(x, u)=\innerproduct{f}{u},\;\; & \forall\,u\in \mathcal D(\overline{q}_E)\,,\\[4pt]
\overline{b}_E(z_-,v_-)= -\innerproduct{f}{v_-},\;\; &\forall\,v_-\in \mathcal D\big(\,\overline{b}\big)\,.
\end{array} \right.
\ee
We recall the identity \eqref{egalite-special}, which is a special case of formula \eqref{egalite-operator} stated in Proposition~\ref{Prop:formulaoperator}: if $(x, z_-) \in \mathcal D(\overline{q}_E)\,\times\, \mathcal D\big(\,\overline{b}\big)$ then, for all $U\in F$, one has
\[ \innerproduct{X}{(A-E)U}=  \overline q_E\(x, u\)- \overline{b}_E(z_-,v_-)\]
with $X=x+z_-$, $u=\Lambda_+ U + L_E\Lambda_+U$ and $v_-=U-u$.
Thanks to this identity, we see that for any solution $(x, z_-)$ of~\eqref{Pf}, the sum $X=x+z_-$ satisfies
\[
\innerproduct{X}{(A-E)U}= \innerproduct{f}{U}\,,\quad\forall U\in F\,.
\]
As a consequence, $X$ is in $\mathcal D(A^*)$ and solves $(A^*-E) X=f$. But this vector is also in $\mathcal D(\overline{q}_E)\,+\, \mathcal D\big(\,\overline{b}\big)\,$, so it solves $(\widetilde A-E)X=f\,$.

\noindent
On the other hand, we can rewrite \eqref{Pf} in terms of the Friedrichs extensions $T_E$ and $B$:\\[4pt]
{\it \centerline{Find $(x, z_-) \in \mathcal D(T_E)\,\times\, \mathcal D\big(B\big)$ such that}}
\begin{equation}
\left\{\begin{array}{ll}
T_E x=\Pi_E(f)\,,\\
(B+E)z_-=-\Lambda_-(f)\,.
\end{array} \label{strong}\right.
\end{equation}
Since $E>\lambda_0$, the operator $B+E$ is invertible with bounded inverse, and by Proposition \ref{Prop:thm1-abis} the same is true with $T_E$ if $E$ is in $(\lambda_0,\lambda_\infty)\setminus\{\lambda_k\,,\,k\geq k_0\}\,$. Then \eqref{strong} has a unique solution given by
\begin{equation}
\left\{\begin{array}{ll}
x=T_E^{-1}\circ \Pi_E(f)\,,\\
z_-=-\,(B+E)^{-1}\circ \Lambda_-(f)\,
\end{array}\label{inversion} \right.
\end{equation}
and the vector $X=\left(T_E^{-1}\circ \Pi_E-(B+E)^{-1}\circ \Lambda_-\right)(f)$ solves $(\widetilde A-E)X=f$.

\noindent
The above discussion shows that for $E$ in $(\lambda_0,\lambda_\infty)\setminus\{\lambda_k\,,\,k\geq k_0\}\,$ the symmetric operator $\widetilde A-E$ is surjective and admits the bounded operator $T_E^{-1}\circ \Pi_E-(B+E)^{-1}\circ \Lambda_-\,$ as a right inverse. But it is well-known that the surjectivity of a symmetric operator implies its injectivity, since its kernel is orthogonal to its range. So $\widetilde A-E$ is invertible and \eqref{inverse} holds true. Another classical result is that a densely defined surjective symmetric operator is always self-adjoint: see, {\it e.g.}, \cite[Corollary 3.12]{schmudgen2012unbounded}. Applying this to $\widetilde A -E\,$, we conclude that $\widetilde A$ is self-adjoint.


The self-adjoint extension $\widetilde A$ is thus built. Its uniqueness among those whose domain is contained in $\mathcal D(\overline{q}_E)\,+\, \mathcal D\big(\,\overline{b}\big)$ is almost trivial. Indeed, if $\hat{A}$ is a self-adjoint extension of $A$, we must have $\mathcal D(\hat{A}) \subset \mathcal D(A^*)$, hence, if in addition $\mathcal D(\hat{A})\subset \mathcal D(\overline{q}_E)\,+\, \mathcal D\big(\,\overline{b}\big)$ then $\mathcal D(\hat{A})\subset \mathcal D\big(\widetilde A\,\big)$, which automatically implies $\hat{A} = \widetilde A$ since both operators are self-adjoint.
This completes the proof of Proposition~\ref{Prop:thm1-b}.
\end{proof}

\subsection{Direct sums}\label{Section: direct sum}

Recall that in~\eqref{def Gamma} we defined the {\it graph} $\Gamma_E$ of $L_E$ as
\[
\Gamma_E:= \big\{x_++L_E\,x_+\,:\, x_+\in F_+\big\}\subset F_+\oplus \mathcal D(B)\,.
\]
A natural question is whether its closure $\overline \Gamma_E$ in $\mathcal H$ has the graph property $\,\overline \Gamma_E\cap \mathcal H_-=\{0\}$. A partial answer to this question is given in the next lemma:
\begin{lemma}\label{Lemma:inclusion} Under conditions~\eqref{H1}-\eqref{H2} and with the above notations,
\be{int1}
\overline \Gamma_E\cap \mathcal H_- \subset \big((B+E)(F_-)\big)^\perp \,.
\ee
\end{lemma}
\begin{proof}
The arguments below are essentially contained in the proof of~\cite[Lemma~2.2]{SST20}, but we repeat them here for the reader's convenience. If $y\in \overline \Gamma_E\cap \mathcal H_-$ then there is a sequence $(x_n)$ in $F_+$ such that $\Vert x_n\Vert\to 0$ and $\Vert L_E x_n-y\Vert\to 0$. Then, for $z\in (B+E)(F_-)$ we may write $\innerproduct{y}{z}=\lim \innerproduct{L_Ex_n}{z}$. On the other hand, \[\vert\innerproduct{L_Ex_n}{z}\vert=\left\vert\innerproduct{x_n}{A\,(B+E)^{-1}z}\right\vert\leq \Vert x_n\Vert\,\left\Vert A\,(B+E)^{-1}z\right\Vert\to 0\,,\]
so $\innerproduct{y}{z}=0$.
\end{proof}
If one assumes as in~\cite{SST20} that $\Lambda_-A\upharpoonright_{_{F_-}}$ is essentially self-adjoint, then the subspace $(B+E)(F_-)$ of $\mathcal H_-$ is dense in $\mathcal H_-$ and one concludes that $\overline \Gamma_E$ has the graph property. But we do not make this assumption, and for this reason we cannot infer from~\eqref{int1} that $\overline \Gamma_E\cap \mathcal H_- =\{0\}$. {\it In other words, we do not know whether the operator $L_E$ is closable or not. This is why we have to resort to a geometric strategy in which the linear subspace $\overline \Gamma_E$ replaces the possibly nonexistent closure of $\,L_E$. Here is the main difference between the present work and~\cite{SST20}.}\medskip

While we may have $\overline \Gamma_E\cap \mathcal H_- \ne\{0\}$, the following property holds true, as a consequence of Lemma~\ref{Lemma:inclusion}:
\begin{proposition}\label{Prop:lem2}
Under conditions~\eqref{H1}-\eqref{H2} and with the above notations,
\[
\overline \Gamma_E \cap \mathcal D\big(\,\overline{b}\big)= \{0\}\,.
\]
\end{proposition}
\begin{proof}
From~\eqref{int1}, we have
\begin{multline*}
\overline \Gamma_E \cap \mathcal D\big(\,\overline{b}\big)= \big(\,\overline \Gamma_E \cap \mathcal H_-\big)\cap\mathcal D\((B+E)^{1/2}\)\\
\subset\big((B+E)(F_-)\big)^\perp \cap \mathcal D\((B+E)^{1/2}\)= (B+E)^{-1/2} \(\(B+E)^{1/2} F_-\)^\perp\)= \{0\}\,,
\end{multline*}
since $(B+E)^{1/2}F_-$ is dense in $\mathcal H_-$.
\end{proof}

Proposition \ref{Prop:lem2} tells us that the sum of $\overline{\Gamma}_E$ and $\mathcal D(\overline{b})$ is algebraically direct. Let us denote by $\pi_E:\,\overline{\Gamma}_E\dot{+} \mathcal D\big(\,\overline{b}\big)\to\overline{\Gamma}_E$ and $\,\pi'_E:\,\overline{\Gamma}_E\dot{+} \mathcal D\big(\,\overline{b}\big)\to\mathcal D\big(\,\overline{b}\big)$ the associated projectors. In Section \ref{Section:min-max} we will need some informations on the continuity of the restrictions $\pi_E\!\upharpoonright_{\mathcal D(\tilde A)}\,$ and $\pi'_E\!\upharpoonright_{\mathcal D(\tilde A)}\,$. These operators are not necessarily continuous for the $\|\cdot\|$ norm, but we have the following result.
\begin{proposition}\label{continuity}
Under assumptions~\eqref{H1}-\eqref{H2}-\eqref{H3}, for all $E>\lambda_0\,$, one has
\[\pi_E\left(\mathcal D(\widetilde A)\right)\subset \mathcal D(T_E)\;\hbox{ and }\;\pi'_E\left(\mathcal D(\widetilde A)\right)\subset \mathcal D(B)\,.\]
As a consequence, the domain of $\widetilde A$ may also be written as
\be{dom bis}
\mathcal D\(\widetilde A\,\)= \big(\mathcal D\(T_E\) \,\dot{+}\, \mathcal D(B)\big)\cap \mathcal D\(A^*\)\,.
\ee
Moreover the operator $\pi_E\!\upharpoonright_{\mathcal D(\tilde A)}\,$ is continuous for the norms $\| \cdot\|_{\mathcal D(\widetilde A)}\,$, $\|\cdot\|_{\mathcal D(T_E)}\,$ and the operator $\pi'_E\!\upharpoonright_{\mathcal D(\tilde A)}\,$ is continuous for the norms $\| \cdot\|_{\mathcal D(\widetilde A)}\,$, $\| \cdot\|_{\mathcal D(B)}\,$. More precisely, there is a positive constant $C_E\,$ such that for all $X\in \mathcal D(\widetilde A)\,$,
\[
\| \pi'_E (X)\|_{\mathcal D(B)}\le C_E\,\| \Lambda_-(\widetilde A -E)X\|\quad\mbox{and}\quad\| \pi_E (X)\|_{\mathcal D(T_E)}\le C_E\,\| X \|_{\mathcal D(\widetilde A)} \,.
\]
The constant $C_E$ remains uniformly bounded when $E$ stays away from $\lambda_0$ and $\,\infty\,$.
\end{proposition}
\begin{proof}
Note that Formula \eqref{inverse} for the inverse of $\widetilde A-E$ already proves the two inclusions $\pi_E\left(\mathcal D(\widetilde A)\right)\subset \mathcal D(T_E)$ and $\pi'_E\left(\mathcal D(\widetilde A)\right)\subset \mathcal D(B)$ when $E$ is in $(\lambda_0,\lambda_\infty)\setminus\{\lambda_k\,:\,k_0\leq k<\infty\}\,$. But we want to prove a statement for {\it all} values of $E$ in $(\lambda_0,\infty)\,$ and this requires some additional work.

In the arguments below, the constant $C_E$ changes value from line to line but we keep the same notation for the sake of simplicity. We shall use the weak form~\eqref{Pf} of the equation $(\widetilde A -E)X=f$ and the equivalent system of strong equations \eqref{strong}, introduced in the proof of Proposition \ref{Prop:thm1-b}. In that proof, $f$ was given, $X=x+z_-$ was unknown and it was shown that for each $E>\lambda_0$ the solvability of~\eqref{Pf} is a sufficient condition for the solvability of $(\widetilde A -E)X=f$. But it turns out that this condition is also necessary. Indeed, taking $E>\lambda_0\,$, $X\in \mathcal D\big(\widetilde A\,\big)\,$ and defining
\[x:=\pi_E(X)\,, \quad z_-:=\pi'_E(X)\,, \quad f:=(\widetilde A -E)X\,,\]
we can apply Formula \eqref{egalite-operator} of Proposition~\ref{Prop:formulaoperator} with the successive choices $U=u\in \mathcal D(\overline{q}_E)\,$, $U=v_-\in  \mathcal D(\overline{b})\,$ and this tells us that $(x,z_-)$ satisfies~\eqref{Pf}. Then, the second equation of the equivalent system \eqref{strong} implies that $z_-=-(B+E)^{-1}\Lambda_-(\widetilde A -E)X$, so $z_-$ is in $\mathcal D(B)$ with an estimate of the form
\[\| z_-\|_{\mathcal D(B)}\le C_E\,\| \Lambda_-(\widetilde A -E)X \| \,.\]
This estimate on $z_-$ implies in turn the estimate $\|x\| \le C_E\,\| X\|_{\mathcal D(\widetilde A)}\,$, since $x=X-z_-\,$, $\| X\|\leq \| X\|_{\mathcal D(\widetilde A)}\,$ and $\| z_-\|\leq\| z_-\|_{\mathcal D(B)}\,$.
Moreover, the first equation in~\eqref{strong} exactly means that $x$ is in $\mathcal D(T_E)$ and $T_E x=\Pi_E (\widetilde A -E)X$, so we finally get the estimate
\[\|x\|_{\mathcal D(T_E)} \le C_E\,\| X\|_{\mathcal D(\widetilde A)}\,.
\]
We thus have the desired inclusions $\pi_E\left(\mathcal D(\widetilde A)\right)\subset \mathcal D(T_E)$ and $\pi'_E\left(\mathcal D(\widetilde A)\right)\subset \mathcal D(B)\,$, hence
$\mathcal D\big(\widetilde A\,\big)\subset \mathcal D(T_E) \dot{+} \mathcal D(B)\,$. Then, remembering the definition $\mathcal D\big(\widetilde A\,\big)=\big(\mathcal D(\overline{q}_E)\,\dot{+}\,\mathcal D\big(\,\overline{b}\big)\big)\cap\mathcal D\big(A^*\big)\,$ and the inclusions $\mathcal D(T_E)\subset \mathcal D(\overline{q}_E)\,$, $\mathcal D(B)\subset \mathcal D\big(\,\overline{b}\big)\,$, one easily gets \eqref{dom bis}.
This ends the proof of Proposition \ref{continuity}.
\end{proof}
\begin{remark}\label{simplified}
In Section \ref{Section:min-max}, we do not use all the information contained in Proposition~\ref{continuity}: we only need the weaker estimates
\be{continuous}
\| \pi'_E (X)\|\le C_E\,\| \Lambda_-(\widetilde A -E)X\|\quad\mbox{and}\quad\| \pi_E (X)\|\le C_E\,\| X \|_{\mathcal D(\widetilde A)} \,.
\ee
\end{remark}

\subsection{Variational interpretation when \texorpdfstring{$k_0=1$}{k0=1}}\label{Section: variationalinterpretation}

In the special case $k_0=1$, for $\lambda_0<E<\lambda_1$ the quadratic form $\overline{q}_E$ is positive definite as well as $\overline{b}_E$ and the existence and uniqueness of a solution to the weak problem \eqref{Pf} directly follows from the Riesz isomorphism theorem. One can even give an interpretation of \eqref{Pf} that generalizes the minimization principle for the Friedrichs extension of semibounded operators mentioned in the introduction. We describe it in this short subsection, as a side remark.

\noindent Assuming that $E\in(\lambda_0,\lambda_1)\,$ and given $f\in \mathcal H$, let us consider the inf-sup problem
\[
I_{E,f}=\inf_{x_+\in F_+}\, \sup_{y_-\in F_-} \Big\{ \tfrac12\innerproduct{x_++y_-}{(A-E)\,(x_++y_-)} - \innerproduct{f}{x_++y_-}\Big\}\,.
\]
Of course, in general, $I_{E,f}$ is not attained, but using the decomposition \eqref{concave} and replacing $F=F_+\oplus F_-\,$ by the larger space $\mathcal D(\overline{q}_E)\dot{+}\mathcal D\big(\,\overline{b}\big)$, one can transform it into a min-max:
\begin{align*}
I_{E,f}&=\inf_{x_+\in F_+}\, \sup_{z_-\in \mathcal D(B)} \Big\{ \tfrac12\,q_E(x_+ + L_E\,x_+ )-\innerproduct{f}{x_++L_E\,x_+}
-\tfrac12\,\overline{b}_E(z_-)- \innerproduct{f}{z_-}\Big\}\\
&=\, \inf_{x_+\in F_+} \Big\{ \tfrac12\,q_E(x_+ + L_E\,x_+ )-\innerproduct{f}{x_++L_E\,x_+} \Big\}\,-\inf_{z_-\in \mathcal D(B)} \Big\{ \tfrac12\,\overline{b}_E(z_-)+ \innerproduct{f}{z_-}\Big\}\\
&=\, \min_{x\in \mathcal D(\overline{q}_E)}\Big\{\tfrac12\, \overline q_E(x)-(f, x)\Big\}-\min_{z_-\in \mathcal D\big(\,\overline{b}\big)}\Big\{ \tfrac12 \overline{b}_E (z_-)+ \innerproduct{f}{z_-}\Big\}\,.
\end{align*}
Each of these last two convex minimization problems has a unique solution, and the system of Euler-Lagrange equations solved by the two minimizers is just~\eqref{Pf}, so their sum is $X=\(\widetilde A-E\)^{-1}\kern-2ptf$.

\section{The min-max principle}\label{Section:min-max}

In this section, we establish the min-max principle for the eigenvalues of $\widetilde A$ that constitutes the last part of Theorem \ref{Theorem:thm1}:
\begin{proposition}\label{Prop:thm1-c} Under assumptions~\eqref{H1}-\eqref{H2}-\eqref{H3}, for $k\ge k_0$ the numbers $\lambda_k$ satisfying $\lambda_k<\lambda_\infty$ are all the eigenvalues of $\widetilde A$ in the spectral gap $(\lambda_0, \lambda_\infty)$ counted with multiplicity. Moreover one has
\[\lambda_\infty=\inf \big(\sigma_{\!\rm ess}(\widetilde A)\cap (\lambda_0,\infty)\big)\,.\]
\end{proposition}

\noindent Even if our assumptions are weaker and our formalism slightly different, the arguments in the proof of Proposition \ref{Prop:thm1-c} are essentially the same as in~\cite[\S~2]{DES00a} (but some details are missing in that reference) and~\cite[\S~2.6]{SST20}. This proof is based on two facts:
 
 - A relation between the min-max levels $\lambda_k$ and the spectrum of $T_E$ which is provided by the second part of Proposition \ref{Prop:thm1-abis}.
 
 - A relation between the spectra of $T_E$ and $\widetilde A$ which is provided by the next lemma, and whose proof relies on Proposition~\ref{Prop:formulaoperator} and on the estimates \eqref{continuous} of Remark \ref{simplified}.

\begin{lemma}\label{TE equiv A}
Under assumptions~\eqref{H1}~\eqref{H2}-\eqref{H3},
let $E>\lambda_0$ and let $r$ be a positive integer. The two following properties are equivalent:

{\it (i)} For all $\delta>0\,$, $\mathrm{Rank}\big({\mathbbm 1}_{(-\delta,\delta)}(T_E)\big)\ge r\,$.

{\it (ii)} For all $\varepsilon >0\,$, $\mathrm{Rank}\big({\mathbbm 1}_{(E-\varepsilon,E+\varepsilon)}\big(\widetilde A\,\big)\big)\ge r\,$.

In other words: $0\in\sigma_{\!\rm ess}(T_E)$ if and only if $E\in\sigma_{\!\rm ess}(\widetilde A)$; $0\in\sigma_{\!\rm disc}(T_E)$ if and only if $E\in\sigma_{\!\rm disc}(\widetilde A)$ and when this happens they have the same multiplicity; $0\in\rho(T_E)$ if and only if $E\in\rho(\widetilde A)$.
\end{lemma}
\begin{proof}
If {\it (i)} holds true, for each $\delta>0$ there is a subspace $\mathcal X_\delta$ of $\,{\mathcal R}\big({\mathbbm 1}_{(-\delta,\delta)}(T_E)\big)$ of dimension $r\,$ (we recall the notation $\,{\mathcal R}(L)$ for the range of an operator $L$). Then we have $\mathcal X_\delta\subset \mathcal D(T_E)\subset \mathcal D(\overline{q}_E)$. Using Proposition~\ref{Prop:formulaoperator} and the second estimate of \eqref{continuous} we find that for all $x\in \mathcal X_\delta$ and $Y\in \mathcal D\big(\widetilde A\,\big)\,$,
\[ \left|\innerproduct{x}{(\widetilde A-E)Y}\right|=\left|\overline{q_E}(x,\pi_E(Y))\right|=\left|\innerproduct{T_E x}{\pi_E(Y)}\right|\le \delta\|x\|\,\| \pi_E(Y)\|\le C_E\,\delta\|x\|\,\| Y\|_{\mathcal D(\widetilde A)}\,.\]
Assume, in addition, that the property {\it (ii)} does not hold true. This means that for some $\varepsilon_0>0\,$, $\mathrm{Rank}\big({\mathbbm 1}_{(E-\varepsilon_0,E+\varepsilon_0)}\big(\widetilde A\,\big)\big) \le r-1\,$. Then for each $\delta>0$ there is $x_\delta$ in $\mathcal X_\delta$ such that $\|x_\delta\|=1$ and ${\mathbbm 1}_{(E-\varepsilon_0,E+\varepsilon_0)}\big(\widetilde A\,\big) x_\delta=0\,$. So there is $Y_\delta\in \mathcal D\big(\widetilde A\,\big)\,$ such that $(\widetilde A - E ) Y_\delta = x_\delta$ and
$\| Y_\delta \| \le \varepsilon_0^{-1}\,$. We thus get $\innerproduct{x_\delta}{\big(\widetilde A-E\big)Y_\delta}=\|x_\delta\|^2=1$ and $C_E\|x_\delta\| \| Y_\delta\|_{\mathcal D(\widetilde A)}$ is bounded independently of $\delta$. So, taking $\delta$ small enough we obtain $\big|\innerproduct{x_\delta}{\big(\widetilde A-E\big)Y_\delta}\big| > C_E\,\delta\|x_\delta\|\,\| Y_\delta\|_{\mathcal D(\widetilde A)}$ and this is absurd. We have thus proved by contradiction that {\it (i)} implies {\it (ii)}.\smallskip

It remains to show that {\it (ii)} implies {\it (i)}. If {\it (ii)} holds true, then for each $\varepsilon>0$ there is a subspace $\mathcal Y_\varepsilon$ of $\,{\mathcal R}\left({\mathbbm 1}_{(E-\varepsilon,E+\varepsilon)}\big(\widetilde A\,\big)\right)$ of dimension $r$ and we have $\mathcal Y_\varepsilon\subset \mathcal D\big(\widetilde A\,\big)\subset \mathcal D(T_E)\dot{+}\mathcal D(B)$. Using Proposition~\ref{Prop:formulaoperator} we find that for all $x\in \mathcal D(T_E)$ and $Y\in \mathcal Y_\varepsilon\,$,
\[
\big|\innerproduct{T_E x}{\pi_E (Y)}\big|=\big|\overline{q_E}\big(x,\pi_E(Y)\big)\big|=\big|\innerproduct{x}{\big(\widetilde A-E\big)Y}\big|\le \varepsilon\,\|x\|\,\| Y\|\,.
\]
Moreover for all $Y\in \mathcal Y_\varepsilon\,,$ from the first estimate of \eqref{continuous} one has
\[
\big\| \pi'_E(Y)\big\|\le C_E\,\big\| \Lambda_-\big(\widetilde A-E\big)Y\big\|\le C_E\,\varepsilon\,\| Y\|\,.
\]
So, imposing $\varepsilon\le \frac1{2\,C_E}\,$ and using the triangular inequality, we get the estimate $\| Y\|\le 2\,\| \pi_E(Y)\|\,$ for all $Y\in \mathcal Y_\varepsilon\,$.
As a consequence, the subspace $V_\varepsilon:=\pi_E (\mathcal Y_\varepsilon)\subset \mathcal D(T_E)$ is $r$-dimensional and for all $x\in \mathcal D(T_E)$ and $y\in V_\varepsilon\,$, one has
\[
\left|\innerproduct{T_E x}{y}\right|\le 2\,\varepsilon\,\|x\|\,\| y\|\,.
\]
Assume, in addition, that {\it (i)} does not hold true. This means that there exists $\delta_0>0\,$ such that $\mathrm{Rank}\big({\mathbbm 1}_{(-\delta_0,\delta_0)}(T_E)\big) \le r-1\,$. Then for each small $\varepsilon$ there is $y_\varepsilon$ in $V_\varepsilon$ such that $\| y_\varepsilon\,\|=1$ and ${\mathbbm 1}_{(-\delta_0,\delta_0)}(T_E) y_\varepsilon=0\,$. So there is $x_\varepsilon\in \mathcal D(T_E)\,$ such that $T_E x_\varepsilon = y_\varepsilon$ and
$\|x_\varepsilon \| \le \delta_0^{-1}\,$. We thus get $\innerproduct{T_E x_\varepsilon}{y_\varepsilon}=\| y_\varepsilon\,\|^2=1$ and $\|x_\varepsilon\,\|\,\| y_\varepsilon\,\|\le \delta_0^{-1}\,$. So, taking $\varepsilon$ small enough we get $\big|\innerproduct{T_E x_\varepsilon}{y_\varepsilon}\big| > 2\,\varepsilon\,\|x_\varepsilon\,\|\,\| y_\varepsilon\,\|$ and this is absurd. We have thus proved by contradiction that {\it (ii)} implies {\it (i)}, so the two properties are equivalent.

Now, given $E>\lambda_0$, $0$ is in $\sigma_{\!\rm ess}(T_E)$ if and only if {\it (i)} holds true for every $r$, and this is equivalent to saying that {\it (ii)} holds true for every $r$, which exactly means that $E\in \sigma_{\!\rm ess}(\widetilde A)$. Similarly, we can say that $0$ is in $\sigma_{\!\rm disc}(T_E)$ and has multiplicity $\mu_E$ as an eigenvalue if and only if {\it (i)} holds true for $\mu_E$ but not for $\mu_E-1$, and this is equivalent to saying that {\it (ii)} holds true for $\mu_E$ but not for $\mu_E-1$, which exactly means that $E\in \sigma_{\!\rm disc}(\widetilde A)$ with multiplicity $\mu_E$. The last statement on $\rho(\widetilde A)$ follows immediately, since for any operator $L$, $\sigma_{\!\rm ess}(L)$, $\sigma_{\!\rm disc}(L)$ and $\rho(L)$ form a partition of $\,\mathbb{C}$. This ends the proof of the lemma.
\end{proof}

\noindent{\it Proof of Proposition~\ref{Prop:thm1-c}}. Let us define  \[\underline{\lambda}:=\inf \big(\sigma_{\!\rm ess}\big(\widetilde A\big)\cap (\lambda_0,\infty)\big)\in [\lambda_0,\infty]\,.\]

By Proposition~\ref{Prop:thm1-abis}, if $E\in (\lambda_0,\lambda_\infty)$ then $0$ is either an element of $\rho(T_E)$ or an eigenvalue of $T_E$ of finite multiplicity $\mu_E$. The second case occurs when $E=\lambda_k$ for some positive integer $k$. Then $\mu_E={\rm card}\{k'\,:\, \lambda_{k'}=\lambda_{k}\}$. So, by Lemma~\ref{TE equiv A}, $(\lambda_0,\lambda_\infty)\cap\sigma_{\!\rm ess}\big(\widetilde A\big)$ is empty hence $\lambda_\infty\leq \underline{\lambda}$, and the levels $\lambda_k$ in $(\lambda_0,\lambda_\infty)$ are all the eigenvalues of $\widetilde A$ in this open interval, counted with multiplicity.

It remains to prove that $\underline{\lambda}\leq \lambda_\infty$. The nontrivial case is when the sequence $(\lambda_k)$ is bounded, so that $\lambda_\infty\in (\lambda_0,\infty)$. If the sequence is nonstationary and bounded, then $\{\lambda_k\,:\,k\geq k_0\}$ is an infinite subset of $\sigma_{\!\rm disc}(\widetilde A)$, so its limit point $\lambda_\infty$ is in $\sigma_{\!\rm ess}(\widetilde A)$. If the sequence is stationary, let $k$ be such that $\lambda_k=\lambda_\infty$. Then, by Proposition~\ref{Prop:thm1-abis}, $0\in\sigma_{\!\rm ess}(T_{\lambda_\infty})$ so, by Lemma~\ref{TE equiv A}, we find once again that $\lambda_\infty\in \sigma_{\!\rm ess}(\widetilde A) $. In conclusion, one always has $\underline{\lambda}\leq \lambda_\infty$ and this ends the proof of Proposition~\ref{Prop:thm1-c}. \finprf

\noindent{\it Proof of Theorem~\ref{Theorem:thm1}}.   Propositions~\ref{Prop:thm1-b}, \ref{Prop:lem2} and \ref{Prop:thm1-c} together imply Theorem~\ref{Theorem:thm1}. 
\finprf


\section{Applications to Dirac-Coulomb operators} \label{Section:DC}

\noindent In this section, we consider the three-dimensional {\it Dirac-Coulomb operator} $D_V=D+V$ mentioned in the Introduction. We assume that $V$ is a linear combination of Coulomb potentials $\vert x-x_j\vert^{-1}\,$ due to $J$ distinct point-like charges located at $x_1,\cdots,x_J$. If we define $D_V$ on the minimal domain $F=C^\infty_c(\R^3\setminus\{x_1,\cdots,x_J\}, \C^4)$, it is obviously symmetric in the Hilbert space $\mathcal H=L^2(\R^3,\C^4)$.
Thanks to Theorem~\ref{Theorem:thm1} we are going to construct a distinguished self-adjoint realization of $D_V$ and give a min-max principle for its eigenvalues, under some conditions on the coefficients of the linear combination. In each case, Assumption \eqref{H3} will be checked by the method of Remark \ref{check H3}.

\subsection{The attractive case}\label{subsection: attractive}
In this subsection we assume that $V(x)=-\sum_{j=1}^J\frac{\nu_j}{|x-x_j|}$ is an attractive potential generated by $J$ distinct point-like nuclei, each having ~$Z_j$ protons with $0<Z_j \leq Z_*\approx137.04$ so that $0<\nu_j=Z_j/Z_*\leq 1$ (we allow non-integer values of $Z_j$). We are going to use Talman's splitting $\Lambda_+\psi=\begin{pmatrix} \phi \\ 0\end{pmatrix}\,$, $\Lambda_-\psi=\begin{pmatrix} 0 \\ \chi\end{pmatrix}\,$ of four-spinors $\psi=\begin{pmatrix} \phi \\ \chi\end{pmatrix}$ into upper and lower two-spinors, also called large and small two-components. Then $\Lambda_+F=\mathfrak{F}\times\{0\}$ and $\Lambda_-F=\{0\}\times\mathfrak{F}$ with $\mathfrak{F}:=C^\infty_c(\R^3\setminus\{x_1,\cdots,x_J\}, \C^2)$.
With the standard notation $\sigma=(\sigma_1,\,\sigma_2,\,\sigma_3)$ for the collection of Pauli matrices, we recall (see \cite{thaller}) that
\[D_V\begin{pmatrix} \phi \\ \chi\end{pmatrix}=\begin{pmatrix} -i\sigma\cdot\nabla\chi+(1+V)\phi \\ -i\sigma\cdot\nabla\phi-(1-V)\chi\end{pmatrix}\,.\]
Assumptions~\eqref{H1}-\eqref{H2} are easily checked with $\lambda_0=-1$.
It remains to check \eqref{H3}. By Remark \ref{check H3}, it suffices to show that for some $k_0\geq 1$, $\ell_{k_0}(0)\geq 0$. Indeed, this inequality implies that $\lambda_{k_0}\geq 0>\lambda_0$. So we are led to study the quadratic form $q_0$. For $\phi\in \mathfrak{F}$ and $\psi_+=\begin{pmatrix} \phi \\ 0\end{pmatrix}$, the quantity $q_0(\psi_++L_0\psi_+)$ is a function of $\phi,\,V$ and in the rest of the subsection it is more convenient to denote it by $\mathfrak{q}^V(\phi)$. With this notation we have
\be{q0}\mathfrak{q}^V(\phi)=\int_{\R^3}\left\{\frac{\vert \sigma\cdot\nabla\phi\vert^2}{1-V}+\left(1+V\right)\vert\phi\vert^2\right\}\,,\quad \forall \phi\in\mathfrak{F}\,.
\ee

We start by the potential $V(x)=-\nu|x|^{-1}\,$ with $0<\nu\leq 1\,$, corresponding to a unique point-like nucleus. We recall the Hardy-Dirac inequality
\be{Hardy-Dirac}
\mathfrak{q}^{-|\cdot|^{-1}}(\phi)=\int_{\R^3}\left\{\frac{\vert \sigma\cdot\nabla\phi\vert^2}{1+\vert x\vert^{-1}}+\left(1-\vert x\vert^{-1}\right)\vert\phi\vert^2\right\}\,\geq\, 0
\,,\quad \forall \phi\in C^\infty_c(\R^3\setminus\{0\}, \C^2)\ee
proved in \cite{DES00a,DELV2004}. Since $\mathfrak{q}^{-\nu|\cdot|^{-1}}\geq\, \mathfrak{q}^{-|\cdot|^{-1}}$, \eqref{Hardy-Dirac} implies that $\ell_1(0)\geq 0$ and Assumption \eqref{H3} is satisfied with $k_0=1$. Then, using Theorem~\ref{Theorem:thm1} we can define a distinguished self-adjoint extension of $D_V$ for $0< \nu\le 1$ and we can also characterize all the eigenvalues of this extension in the spectral gap $(-1,1)$ by the min-max principle~\eqref{min-max}. This is not a new result: see~\cite{DES00a,EstLos-07,ELS19,SST20}, and it is known that $V$ can be replaced by more general attractive potentials that are bounded from below by $-|x|^{-1}$.\smallskip

We now assume that $J\geq 2$. In such a case, the distinguished self-adjoint extension was constructed in~\cite{Nen77,Klaus-80b} in the subcritical case $\nu_i<1\,(\forall i)$ by a method completely different from the one considered in the present work. Talman's min-max principle for the eigenvalues of the extension was studied in~\cite{ELS21}, also in the subcritical case. But that paper appealed to the abstract result of~\cite{DES06} and as mentioned in the Introduction, the arguments in~\cite{DES06} suffered from the same closability issue as~\cite{DES00a}. Theorem~\ref{Theorem:thm1} solves this issue, moreover it provides a unified treatment: construction of the extension and justification of the min-max principle even in the critical case, {\it i.e.}, when some of the coefficients $\nu_i$ are equal to $1$. But of course, in order to apply this theorem we have to check \eqref{H3} and this is more delicate than in the one-center case. Indeed, when the total number of protons $\sum_jZ_j$ is larger than $137.04\,$, if the nuclei are close to each other one expects some eigenvalues of the distinguished extension to {\it dive} into the negative continuum. If this happens, the corresponding min-max levels $\lambda_k$ should become equal to $\lambda_0$. To check Assumption~\eqref{H3} in such a situation, let us prove by contradiction that for {\it some} $k_0\geq 1$, the inequality $\ell_{k_0}(0)\geq 0$ holds true.

\noindent
Otherwise, there exists a sequence $(G_k)_{k\geq 1}$ of $k$-dimensional subspaces of $\mathfrak{F}$ such that $\mathfrak{q}^V(\phi)<0$ for all $\phi\in G_k\setminus\{0\}$. So one can construct by induction a sequence $(\phi_k)$ of wave functions such that $\phi_k\in G_k$ and $\innerproduct{\phi_k}{\phi_l}_{L^2(\R^3,\C^2)}=\delta_{kl}$. Then $\phi_k$ converges weakly to $0$ in $L^2(\R^3,\C^2)$. In order to derive a contradiction, one can try to prove that for $k$ large enough, $\mathfrak{q}^V(\phi_k)\geq 0\,$. In the subcritical case $\nu_i<1\,(\forall i)$ this has been done in~\cite[Section 6, Step 4, p. 1448-1449]{ELS21}. We give below a proof that is also valid in the critical case. In what follows, the constant $C$ changes from line to line but we keep the same notation for the sake of simplicity.

\noindent
With $\delta:=\frac{1}{2}\min_{1\leq j<j'\leq K} \vert x_j-x_{j'}\vert\,$ one takes $R>\delta +\max_{1\leq j \leq J} \vert x_j\vert$ (to be chosen later) and a partition of unity $(\theta_j)_{0\leq j\leq J+1}$
consisting of smooth functions with values in $[0,1]$ such that $\sum_{j=0}^{J+1} \theta_j^2=1$, ${\rm supp}(\theta_0)\subset B(0,2R)\setminus\bigcup_{j=1}^JB(x_j,\delta/2)$, ${\rm supp}(\theta_j)\subset B(x_j,\delta)$ for $1\leq j\leq J$ and ${\rm supp}(\theta_{J+1})\cap B(0,R)=\emptyset$.
The pointwise IMS formula \cite[Lemma 4.1]{bosi} for the Pauli operator gives
\[\vert\sigma\cdot\nabla\phi_k\vert^2=\sum_{j=0}^{J+1}\vert\sigma\cdot\nabla(\theta_j\phi_k)\vert^2-\left(\sum_{j=0}^{J+1}\vert\nabla \theta_j\vert^2\right)\vert\phi_k\vert^2\,,
\]
so, remembering that $\Vert\phi_k\Vert_{L^2(\R^3)}^2=1\,$, one gets
\begin{align*}
\mathfrak{q}^V(\phi_k)&=\sum_{j=0}^{J+1}\mathfrak{q}^V(\theta_j\phi_k)-\int_{\R^3} \left(\sum_{j=0}^{J+1}\vert\nabla \theta_j\vert^2\right)\frac{\vert\phi_k\vert^2}{1-V}\\
&=\frac{1}{2}+\sum_{j=0}^J\mathfrak{q}^V(\theta_j\phi_k)+\left(\mathfrak{q}^V(\theta_{J+1}\phi_k)-\frac{1}{2}\Vert\theta_{J+1}\phi_k\Vert_{L^2(\R^3)}^2\right)\\
&\qquad\qquad\qquad\qquad\;\;\; -\int_{\R^3} \left(\frac{1-\theta_{J+1}^2}{2}+\frac{1}{1-V}\sum_{j=0}^{J+1}\vert\nabla \theta_j\vert^2\right)\vert\phi_k\vert^2\\
&\geq \frac{1}{2}+\sum_{j=0}^{J}\mathfrak{q}^V(\theta_j\phi_k)+\left(\mathfrak{q}^V(\theta_{J+1}\phi_k)-\frac{1}{2}\Vert\theta_{J+1}\phi_k\Vert_{L^2(\R^3)}^2\right) -C\int_{B(0,2R)}\vert\phi_k\vert^2\,.
\end{align*}
From now on, we fix $R$ such that $-V\leq 1/4$ on $\R^3\setminus B(0,R)$. Then one has
\[\mathfrak{q}^V(\theta_{J+1}\phi_k)-\frac{1}{2}\Vert\theta_{J+1}\phi_k\Vert_{L^2(\R^3)}^2\geq \frac{1}{4}\Vert \theta_{J+1}\phi_k\Vert_{H^1(\R^3)}^2\,.
\]
Let
\[M:=1+\max\left\{\sup_{x\in {\rm supp}(\theta_0)}\kern-13pt-V(x)\;\;;\;\sup_{x\in {\rm supp}(\theta_1)}\kern-13pt(-V(x)-\vert x-x_1\vert^{-1})\;\;;\cdots;\;\sup_{x\in {\rm supp}(\theta_J)}\kern-13pt(-V(x)-\vert x-x_J\vert^{-1})\right\}\,.\]
Then
\[\mathfrak{q}^V(\theta_0\phi_k)\geq \frac{1}{M}\Vert \theta_{0}\phi_k\Vert_{H^1(\R^3)}^2-C\int_{\R^3}\vert\theta_{0}\phi_k\vert^2\]
and, introducing the rescaled functions   $\hat{\phi}_{j,k}(y):=(\theta_j\phi_k)(x_j+M^{-1}y)$ for $1\leq j\leq J$, one finds
\[\mathfrak{q}^V(\theta_j\phi_k)\geq \frac{1}{M^2}\mathfrak{q}^{-\vert \cdot\vert^{-1}}(\hat{\phi}_{j,k})-C\int_{\R^3}\vert\theta_{j}\phi_k\vert^2\,.\]
Gathering these estimates, one gets the lower bound
\be{lower}\mathfrak{q}^V(\phi_k)\geq \frac{1}{2}+\frac{1}{M^2}\sum_{j=1}^J \mathfrak{q}^{-\vert \cdot\vert^{-1}}(\hat{\phi}_{j,k}) + \frac{1}{M}\Vert \theta_{0}\phi_k\Vert^2_{H^1(\R^3)}+\frac{1}{4}\Vert \theta_{J+1}\phi_k\Vert^2_{H^1(\R^3)}-C\int_{B(0,2R)}\vert\phi_k\vert^2\,.\ee
From the Hardy-Dirac inequality \eqref{Hardy-Dirac}, each of the terms $\mathfrak{q}^{-\vert \cdot\vert^{-1}}(\hat{\phi}_{j,k})$ is nonnegative, so the assumptions that $\Vert\phi_k\Vert_{L^2(\R^3)}=1$ and
$\mathfrak{q}^V(\phi_k)<0$ imply that the quantities $\,\mathfrak{q}^{-\vert \cdot\vert^{-1}}(\hat{\phi}_{j,k})\,$, $\Vert\theta_{0}\phi_k\Vert_{H^1}$ and $\Vert\theta_{J+1}\phi_k\Vert_{H^1}$ are uniformly bounded. But from \cite[Theorem 1.9]{ELS19}, for $0\leq s<1/2$ there is a positive constant $\kappa_s$ such that
\[\mathfrak{q}^{-\vert \cdot\vert^{-1}}(\phi)+\Vert\phi\Vert^2_{L^2(\R^3)}\geq \kappa_s\Vert\phi\Vert^2_{H^s(R^3)}\,,\quad\forall \phi\in\mathfrak{F}\,.
\]
Applying this inequality to the functions $\hat{\phi}_{j,k}$ ($1\leq j\leq J$), one easily finds that the sequence $(\phi_k)_{k\geq 1}$ is bounded in $H^s(\R^3)$, hence precompact in $L^2_{\rm loc}(\R^3)$. Since this sequence converges weakly to zero in $L^2(\R^3)$, one concludes that
\[\lim_{k\to\infty}\, \int_{B(0,2R)}\vert\phi_k\vert^2\,=\,0\,.\]
Combining this information with \eqref{lower} one finds that for $k$ large enough, $\mathfrak{q}^V(\phi_k)\geq 0$ and this is a contradiction.

\noindent
In conclusion, the assumptions of Theorem~\ref{Theorem:thm1} are satisfied in our multi-center example, with $k_0$ possibly larger than $1$.

\subsection{The sign-changing case}\label{subsection: sign-changing}

We now consider a potential of the form
\[
V(x)=-\frac{\nu_1}{|x|}+\frac{\nu_2}{|x-x_0|}\quad\mbox{with}\quad x_0\neq 0\,,\quad0<\nu_1\le 1\quad\mbox{and}\quad 0<\nu_2\le \frac2{\frac\pi2+\frac2\pi}\;.
\]
The corresponding Dirac-Coulomb operator $D_V$ is obviously symmetric if we define it on the ``minimal" domain $C^\infty_c(\R^3\setminus\{0,x_0\}, \C^4)$.
But Talman's decomposition in upper and lower spinors cannot be used: due the unbounded repulsive term $\frac{\nu_2}{\vert x-x_0\vert}\,$, \eqref{H2} would not be satisfied. Instead, for the splitting we choose the free-energy projectors
\[
\Lambda_\pm=\mathbbm 1_{\,\R_{\pm}}(D)\,.
\]
We recall (see \cite{thaller}) that
\[D\Lambda_\pm=\Lambda_\pm D=\pm\sqrt{1-\Delta}\,\Lambda_\pm=\pm\Lambda_\pm\sqrt{1-\Delta}\;.\]
In momentum space ({\it i.e.} after Fourier transform), $\Lambda_\pm$ becomes the multiplication operator by the matrix \[M_\pm(p)=\frac{1}{2}\left(I_4\pm \frac{\alpha\cdot p+\beta}{\sqrt{\vert p\vert^2+1}}\right)\,.\]
This matrix depends smoothly on $p$ and is bounded on $\R^3$ as well as its derivatives. As a consequence, the multiplication by $M_\pm$ preserves the Schwartz class $\mathcal S(\R^3,\mathbb{C}^4)$. So the same is true for $\Lambda_\pm$ in position space. But this nonlocal operator does not preserve the compact support property, so \eqref{H1} does not hold for the domain $C^\infty_c(\R^3\setminus\{0,x_0\}, \C^4)$. Since $\Lambda_+C^\infty_c(\R^3\setminus\{0,x_0\}, \C^4)\subset \mathcal S(\R^3,\mathbb{C}^4)\subset\mathcal D(\overline{D_V})$, one can either replace the minimal domain by $F=\Lambda_+C^\infty_c(\R^3\setminus\{0,x_0\}, \C^4)\oplus \Lambda_-C^\infty_c(\R^3\setminus\{0,x_0\}, \C^4)$ as mentioned in the first comment after Theorem~\ref{Theorem:thm1}, or by $F=\mathcal S(\R^3,\mathbb{C}^4)$. In what follows, $A$ is the restriction of $\overline{D_V}$ to one of these two domains. We do not need to specify which one: the arguments proving \eqref{H2}-\eqref{H3} are the same in both cases.

By the upper bound on $\nu_2$, it follows from an inequality of Tix~\cite{MR1608118} that the Brown-Ravenhall operator $-\Lambda_- \big(A+1-\nu_2\big)\upharpoonright_{_{F_-}}=\Lambda_-(\sqrt{1-\Delta}-V-1+\nu_2)\upharpoonright_{_{F_-}}\,$ is non-negative, so~\eqref{H2} holds true with $\lambda_0\leq -1+\nu_2$. In order to bound $\lambda_1$ from below, we can use~\cite[inequality~(38)]{DES00a}. This inequality involves a parameter $\nu\in (0,1)$ and is stated for all functions $\psi_+\in F_+$. One easily passes to the limit $\nu\to 1$ with $\psi_+$ fixed and this gives us the inequality
\be{eq:Hardy-type-free}\innerproduct{\psi_+}{\sqrt{1-\Delta}\,\psi_+}_{L^2(\R^3)}-\,\int_{\R^3}\frac{|\psi_+|^2}{|x|}
\,+\,\innerproduct{\Lambda_-\frac{1}{|x|}\psi_+}{ \left(B_{-|x|^{-1}}\right)^{-1}\Lambda_-\frac{1}{|x|}\psi_+}_{L^2(\R^3)}\,\geq\,0
\ee
for all $\psi_+\in F_+$. Here, we denote by $B_{\mathcal V}$ the Friedrichs extension of the Brown-Ravenhall operator $\Lambda_-(\sqrt{1-\Delta}-\mathcal V)\upharpoonright_{_{F_-}}$, for any electric potential $\mathcal V$ such that this operator is boun\-ded from below.
Inequality \eqref{eq:Hardy-type-free} exactly says that if one chooses $(\nu_1,\nu_2)=(1,0)$ then there holds $q_0(\psi_++L_0\psi_+)\geq 0$ for all $\psi_+\in F_+$, so $\ell_1(0)\geq 0\,$, hence $\lambda_1\geq 0>\lambda_0\,$. This remains true for $0<\nu_1\leq 1$ and $0 < \nu_2\leq \frac{2}{\pi/2+2/\pi}$, since the min-max level $\lambda_1$ is a non-decreasing function of $V$.
Thus, Theorem~\ref{Theorem:thm1} can be applied with $k_0=1$ in order to find a distinguished self-adjoint extension of $D_V$ and to characterize its eigenvalues by a min-max principle.

Note that by~\cite[Corollary~3]{tix1997self}, the operator $-\Lambda_- A\upharpoonright_{_{F_-}}$ is not essentially self-adjoint for $\nu_2 >3/4$. So the abstract result~\cite[Theorem 1.1]{SST20} cannot be applied in this case.

\bigskip
\noindent{\bf Acknowledgment:} The authors wish to thank the referees for their very careful reading of the manuscript, and for many comments and suggestions that greatly improved the quality of this paper.\smallskip

\noindent J.D.~has been partially supported by the Project EFI (ANR-17-CE40-0030) and M.J.E. and E.S. by the project molQED (ANR-17-CE29-0004) of the French National Research Agency (ANR). \\\noindent{\scriptsize\copyright\,\the\year~by the authors. This paper may be reproduced, in its entirety, for non-commercial purposes.}



\bigskip\begin{center}\rule{2cm}{0.5pt}\end{center}\bigskip
\end{document}